\documentclass{article}

\usepackage[french,british]{babel}

\usepackage[T1]{fontenc}

\usepackage{amsmath}

\usepackage{amsfonts}
\usepackage{amsbsy}
\usepackage{amssymb}

\usepackage{amsthm}

\theoremstyle{plain}

\usepackage{latexsym}
\usepackage{amsmath}
\usepackage{amsfonts}
\usepackage{amssymb}
\usepackage{psfrag}
\usepackage{graphicx}

\usepackage{amsmath,amssymb}
\usepackage{graphicx}

\usepackage{amsmath,amssymb}
\usepackage{graphicx}

\newtheorem{defn}{Definition}
\newtheorem{theorem}{Theorem}
\newtheorem{corollary}{Corollary}
\newtheorem{remark}{Remark}
\newtheorem{lemma}{Lemma}

\newcommand{\eps}{\varepsilon}
\newcommand{\me}{\mathbf}
\newcommand{\mr}{\mathbb}
\newcommand{\mt}{\mathsf}
\newcommand{\md}{\mathcal}
\newcommand{\ld}{\left}
\newcommand{\rd}{\right}

\newcommand{\ip}{\int_{-\pi}^{\pi}}

\newcommand{\be}{\begin{equation}}
\newcommand{\ee}{\end{equation}}

\newcommand{\bem}{\begin{multline}}
\newcommand{\eem}{\end{multline}}

\newcommand{\bml}{\begin{multline*}}
\newcommand{\eml}{\end{multline*}}

\newcommand{\beg}{\begin{gather}}
\newcommand{\eeg}{\end{gather}}

\begin{document}

\title{Filtering problem for sequences with periodically stationary multiseasonal  increments with spectral densities allowing canonical factorizations}

\author{
Maksym Luz\thanks {BNP Paribas Cardif in Ukraine, Kyiv, Ukraine, maksym.luz@gmail.com},
Mikhail Moklyachuk\thanks
{Department of Probability Theory, Statistics and Actuarial
Mathematics, Taras Shevchenko National University of Kyiv, Kyiv 01601, Ukraine, moklyachuk@gmail.com}
}

\date{\today}

\maketitle

\renewcommand{\abstractname}{Abstract}
\begin{abstract}
We consider a stochastic sequence $\xi(m)$  with periodically stationary generalized multiple increments of fractional order which combines cyclostationary, multi-seasonal, integrated and fractionally integrated patterns. The filtering problem is solved for this type of sequences  based on observations  with a periodically stationary noise.
When spectral densities are known and allow the canonical factorizations, we derive  the mean square error and the spectral characteristics of the optimal estimate of the functional $A{\xi}=\sum_{k=0}^{\infty}{a}(k) {\xi}(-k)$.
Formulas that determine the least favourable spectral densities and the minimax (robust) spectral
characteristics of the optimal linear estimate of the functional
are proposed in the case where the spectral densities
are not   known, but some sets of admissible spectral densities are given.
\end{abstract}

\maketitle


\maketitle

\vspace{2ex}

\section*{Introduction}
Non-stationary  time series models have found wide-ranging applications in  economics, finance, climatology, air pollution,   signal processing. A fundamental example is the general multiplicative model, known as $SARIMA (p, d, q)\times(P, D, Q)_s$, which was introduced in the  book by Box and Jenkins et al. \cite{Box_Jenkins}. This model incorporates both integrated and seasonal factors, and can be represented as follows:
\be
 \Psi (B ^ s) \psi (B) (1-B) ^ d (1-B ^ s) ^ Dx_t = \Theta (B ^ s) \theta (B) \eps_t, \label {seasonal_3_model} \ee
where ${\eps_t}$ is a sequence of independent and identically distributed (i.i.d.) random variables, and where $ \Psi (z) $ and $ \Theta (z) $ are two polynomials of degrees $ P $ and $ Q $, respectively, with roots outside the unit circle. The parameters $d$ and $D$ can take fractional values. The process in equation \eqref{seasonal_3_model} is stationary and invertible when $|d+D|<1/2$ and $|D|<1/2$.
One application of seasonal ARFIMA models to the analysis of monetary aggregates used by the U.S. Federal Reserve is demonstrated in the work of Porter-Hudak \cite{Porter-Hudak}.

In the field of statistical inference for seasonal long-memory sequences, recent research has yielded several notable results. One such contribution is the work by Tsai, Rachinger, and Lin  \cite{Tsai}, who developed methods for estimating model parameters when measurement errors are present.
In another study, Baillie, Kongcharoen, and Kapetanios  \cite{Baillie} compared two commonly used estimation procedures for prediction problem based on  ARFIMA models. Specifically, they compared the performance of maximum likelihood estimation (MLE) to  a two-step local Whittle estimator. Through a   simulation study, they found that the MLE estimator   outperformed the two-step local Whittle estimator.
In addition, Hassler and Pohle  \cite{Hassler}  evaluated the predictive performance of various forecasting methods for inflation and return volatility time series. Their analysis provided compelling evidence in support of models with a fractional integration component.

Another class of non-stationary processes is the periodically correlated or cyclostationary processes, introduced by Gladyshev \cite{Gladyshev}. These processes belong to the class of time-dependent spectrum processes and are widely used in signal processing and communications. For recent works on cyclostationarity and its applications, see the review by Napolitano \cite{Napolitano}. Periodic time series can be viewed as an extension of seasonal models \cite{Baek,Basawa,Lund,Osborn}.

The methods used for parameter estimation and filtering of time series data often fail to account for real-world challenges such as outliers, measurement errors, incomplete information about spectral  structure. As a result, there is a growing interest in robust estimation methods that can effectively handle such issues. For example, Reisen et al. \cite{Reisen2018} and Solci et al. \cite{Solci} have proposed robust estimates for SARIMA and PAR models. Other researchers, including Grenander \cite{Grenander}, Hosoya \cite{Hosoya},  Franke \cite{Franke1985}, Vastola and Poor \cite{VastPoor1984}, Moklyachuk \cite{Moklyachuk,Moklyachuk2015}, and Luz and Moklyachuk \cite{Luz_Mokl_filt3}, Liu et al. \cite{Liu}, have also investigated various aspects of minimax extrapolation, interpolation, and filtering problems for stationary sequences and processes.

In this article, we extend our investigation of robust filtering for stochastic sequences with periodically stationary long memory multiple seasonal increments (or sequences with periodically stationary general multiplicative (GM) increments) by focusing on spectral densities that allow canonical factorizations, whereas in \cite{Luz_Mokl_filt3}, the results were obtained using Fourier transformations of the spectral densities.

The  mentioned sequences were introduced by Luz and Moklyachuk in our earlier work \cite{Luz_Mokl_extra_GMI}, motivated by an increasing interest in models with multiple seasonal and periodic patterns (see the works of Dudek \cite{Dudek}, Gould et al. \cite{Gould},  and Hurd and Piparas \cite{Hurd}). This research continues previous works on minimax filtering of stationary vector-valued processes, periodically correlated processes, and processes with stationary increments. Specifically, Moklyachuk and Masyutka \cite{Mokl_Mas_filt}, Moklyachuk and Golichenko (Dubovetska) \cite{Dubovetska_filt}, and Luz and Moklyachuk \cite{Luz_Mokl_book} have performed research in these areas. Additionally, we mention the works by Moklyachuk, Masyutka, and Sidei \cite{Sidei_book}, which derive minimax estimates of stationary processes from observations with missing values.

The article is structured as follows. In Section $\ref{spectral_ theory}$, we provide a brief review of the   GM increment sequence $\chi_{\overline{\mu},\overline{s}}^{(d)}(\vec{\xi}(m))$ and the stochastic sequence $\xi(m)$ with periodically stationary (periodically correlated, cyclostationary) GM increments, as well as the spectral theory of vector-valued GM increment sequences.
In Section $\ref{classical_filtering}$, we address the classical filtering problem for linear functionals $A\xi$ and $A_N\xi $ that are constructed from unobserved values of the sequence $\xi(m)$. We assume that the spectral densities of the sequence $\xi(m)$ and a noise sequence $\eta(m)$ are known and allow canonical factorization.    The estimates are derived in terms of coefficients of canonical factorizations of the spectral densities, making use of results obtained in \cite{Luz_Mokl_filt_GMI} by using the Fourier transformations  of the spectral densities.
Section $\ref{minimax_filtering}$ focuses on the minimax (robust) estimation for cases where the spectral densities of sequences are not precisely known, but some sets of admissible spectral densities are specified. For illustration, We propose particular types of admissible spectral density sets, which are   generalizations of the sets described in a survey article by Kassam and Poor \cite{Kassam_Poor} for stationary stochastic processes.

\section{Stochastic sequences with periodically stationary generalized multiple increments}\label{spectral_ theory}

\subsection{Preliminary notations and definitions}
 Consider a   stochastic sequence $\xi(m)$, $m\in\mathbb Z$, and a backward shift operator $B_{\mu}$   with the step $\mu\in
\mathbb Z$, such that $B_{\mu}\xi(m)=\xi(m-\mu)$; $B:=B_1$. Then $B_{\mu}^s=B_{\mu}B_{\mu}\cdot\ldots\cdot B_{\mu}$. Define a \emph{multiplicative incremental operator}
\[
\chi_{\overline{\mu},\overline{s}}^{(d)}(B)
=\prod_{i=1}^r(1-B_{\mu_i}^{s_i})^{d_i}
=\sum_{k=0}^{n(\gamma)}e_{\gamma}(k)B^k,
\]
where
$d:=d_1+d_2+\ldots+d_r$, $\overline{d}=(d_1,d_2,\ldots,d_r)\in (\mr N^*)^r$,
 $\overline{s}=(s_1,s_2,\ldots,s_r)\in (\mr N^*)^r$
and $\overline{\mu}=(\mu_1,\mu_2,\ldots,\mu_r)\in (\mr N^*)^r$ or $\in (\mr Z\setminus\mr N)^r$; $n(\gamma):=\sum_{i=1}^r\mu_is_id_i$. Here $\mr N^*=\mr N\setminus\{0\}$.
The explicit representation of the coefficients $e_{\gamma}(k)$ is given in \cite{Luz_Mokl_extra_GMI}.
Within the article, $\delta_{lp}$ denotes Kronecker symbols, ${n \choose l}=\frac{n!}{l!(n-l)!}$.

\begin{defn}[\cite{Luz_Mokl_extra_GMI}]\label{def_multiplicative_Pryrist}
For a stochastic sequence $\xi(m)$, $m\in\mathbb Z$, the
sequence
\begin{eqnarray}
\nonumber
\chi_{\overline{\mu},\overline{s}}^{(d)}(\xi(m))&:=&\chi_{\overline{\mu},\overline{s}}^{(d)}(B)\xi(m)
=(1-B_{\mu_1}^{s_1})^{d_1}(1-B_{\mu_2}^{s_2})^{d_2}\cdot\ldots\cdot(1-B_{\mu_r}^{s_r})^{d_r}\xi(m)
\\&=&\sum_{l_1=0}^{d_1}\ldots \sum_{l_r=0}^{d_r}(-1)^{l_1+\ldots+ l_r}{d_1 \choose l_1}\cdot\ldots\cdot{d_r \choose l_r}\xi(m-\mu_1s_1l_1-\cdots-\mu_rs_rl_r)
\label{GM_Pryrist}
\end{eqnarray}
is called \emph{a stochastic  generalized multiple (GM)  increment sequence} of differentiation   order
$d$
with a fixed seasonal  vector $\overline{s}\in (\mr N^*)^r$
and a varying step $\overline{\mu}\in (\mr N^*)^r$ or $\in (\mr Z\setminus\mr N)^r$.
\end{defn}

\begin{defn}[\cite{Luz_Mokl_extra_GMI}]
\label{oznStPryrostu}
A stochastic GM increment sequence $\chi_{\overline{\mu},\overline{s}}^{(d)}(\xi(m))$  is called   a wide sense
stationary if the mathematical expectations
\begin{eqnarray*}
\mt E\chi_{\overline{\mu},\overline{s}}^{(d)}(\xi(m_0))& = &c^{(d)}_{\overline{s}}(\overline{\mu}),
\\
\mt E\chi_{\overline{\mu}_1,\overline{s}}^{(d)}(\xi(m_0+m))\chi_{\overline{\mu}_2,\overline{s}}^{(d)}(\xi(m_0))
& = & D^{(d)}_{\overline{s}}(m;\overline{\mu}_1,\overline{\mu}_2)
\end{eqnarray*}
exist for all $m_0,m,\overline{\mu},\overline{\mu}_1,\overline{\mu}_2$ and do not depend on $m_0$.
The function $c^{(d)}_{\overline{s}}(\overline{\mu})$ is called a mean value  and the function $D^{(d)}_{\overline{s}}(m;\overline{\mu}_1,\overline{\mu}_2)$ is
called a structural function of the stationary GM increment sequence (of a stochastic sequence with stationary GM increments).
\\
The stochastic sequence $\xi(m)$, $m\in\mathbb   Z$
determining the stationary GM increment sequence
$\chi_{\overline{\mu},\overline{s}}^{(d)}(\xi(m))$ by   \eqref{GM_Pryrist} is called a stochastic
sequence with stationary GM increments (or GM increment sequence of order $d$).
\end{defn}

\begin{remark}
Spectral properties of one-pattern increment sequence $\chi_{\mu,1}^{(n)}(\xi(m)):=\xi^{(n)}(m,\mu)=(1-B_{\overline{\mu}})^n\xi(m)$ and the continuous time increment process $\xi^{(n)}(t,\tau)=(1-B_{\tau})^n\xi(t)$ are described in \cite{Yaglom:1955}, \cite{Yaglom}.
\end{remark}

\subsection{Definition and spectral representation of stochastic sequences with periodically stationary GM increment}

In this subsection, we present definition, justification and a brief review of the spectral theory of stochastic sequences with periodically stationary multiple seasonal increments, introduced in \cite{Luz_Mokl_extra_GMI}.

\begin{defn}
\label{OznPeriodProc}
A stochastic sequence $\xi(m)$, $m\in\mathbb Z$ is called a \emph{stochastic
sequence with periodically stationary (periodically correlated) GM increments} with period $T$ if the mathematical expectations
\begin{eqnarray*}
\mt E\chi_{\overline{\mu},T\overline{s}}^{(d)}(\xi(m+T)) & = & \mt E\chi_{\overline{\mu},T\overline{s}}^{(d)}(\xi(m))=c^{(d)}_{T\overline{s}}(m,\overline{\mu}),
\\
\mt E\chi_{\overline{\mu}_1,T\overline{s}}^{(d)}(\xi(m+T))\chi_{\overline{\mu}_2,T\overline{s}}^{(d)}(\xi(k+T))
& = & D^{(d)}_{T\overline{s}}(m+T,k+T;\overline{\mu}_1,\overline{\mu}_2)
= D^{(d)}_{T\overline{s}}(m,k;\overline{\mu}_1,\overline{\mu}_2)
\end{eqnarray*}
exist for every  $m,k,\overline{\mu}_1,\overline{\mu}_2$ and  $T>0$ is the least integer for which these equalities hold.
\end{defn}

It follows from  Definition \ref{OznPeriodProc} that the sequence
\begin{equation}
\label{PerehidXi}
\xi_{p}(m)=\xi(mT+p-1), \quad p=1,2,\dots,T; \quad m\in\mathbb Z
\end{equation}
forms a vector-valued sequence
$\vec{\xi}(m)=\left\{\xi_{p}(m)\right\}_{p=1,2,\dots,T}, m\in\mathbb Z$
with stationary GM increments as follows:
\[
\chi_{\overline{\mu},\overline{s}}^{(d)}(\xi_p(m))=\chi_{\overline{\mu},T\overline{s}}^{(d)}(\xi(mT+p-1)),\quad p=1,2,\dots,T,
\]
where $\chi_{\overline{\mu},\overline{s}}^{(d)}(\xi_p(m))$ is the GM increment of the $p$-th component of the vector-valued sequence $\vec{\xi}(m)$.

The following theorem describes the spectral structure of the vector-valued GM increment  \cite{Karhunen}, \cite{Luz_Mokl_extra_GMI}.

\begin{theorem}\label{thm1}
1. The mean value and the structural function
 of the vector-valued stochastic stationary
GM increment sequence $\chi_{\overline{\mu},\overline{s}}^{(d)}(\vec{\xi}(m))$ can be represented in the form
\begin{eqnarray}
\label{serFnaR_vec}
c^{(d)}_{ \overline{s}}(\overline{\mu})& = &c\prod_{i=1}^r\mu_i^{d_i},
\\
\label{strFnaR_vec}
 D^{(d)}_{\overline{s}}(m;\overline{\mu}_1,\overline{\mu}_2)& = &\int_{-\pi}^{\pi}e^{i\lambda
m} \chi_{\overline{\mu}_1}^{(d)}(e^{-i\lambda})\chi_{\overline{\mu}_2}^{(d)}(e^{i\lambda})\frac{1}
{|\beta^{(d)}(i\lambda)|^2}dF(\lambda),
\end{eqnarray}
where
\[\chi_{\overline{\mu}}^{(d)}(e^{-i\lambda})=\prod_{j=1}^r(1-e^{-i\lambda\mu_js_j})^{d_j}, \quad \beta^{(d)}(i\lambda)= \prod_{j=1}^r\prod_{k_j=-[s_j/2]}^{[s_j/2]}(i\lambda-2\pi i k_j/s_j)^{d_j},
\]
 $c$ is a vector, $F(\lambda)$ is the matrix-valued spectral function of the stationary stochastic sequence $\chi_{\overline{\mu},\overline{s}}^{(d)}(\vec{\xi}(m))$. The vector $c$
and the matrix-valued function $F(\lambda)$ are determined uniquely by the GM
increment sequence $ \chi_{\overline{\mu},\overline{s}}^{(d)}(\vec \xi(m))$.

2. The stationary vector-valued GM increment sequence $\chi_{\overline{\mu},\overline{s}}^{(d)}(\vec{\xi}(m))$ admits the spectral representation
\begin{equation}
\label{SpectrPred_vec}
\chi_{\overline{\mu},\overline{s}}^{(d)}(\vec{\xi}(m))
=\int_{-\pi}^{\pi}e^{im\lambda}\chi_{\overline{\mu}}^{(d)}(e^{-i\lambda})\frac{1}{\beta^{(d)}(i\lambda)}d\vec{Z}_{\xi^{(d)}}(\lambda),
\end{equation}
where $\vec{Z}_{\xi^{(d)}}(\lambda)=\{Z_{ p}(\lambda)\}_{p=1}^{T}$ is a (vector-valued) stochastic process with uncorrelated increments on $[-\pi,\pi)$ connected with the spectral function $F(\lambda)$ by
the relation
\[
 \mt E(Z_{p}(\lambda_2)-Z_{p}(\lambda_1))(\overline{ Z_{q}(\lambda_2)-Z_{q}(\lambda_1)})
 =F_{pq}(\lambda_2)-F_{pq}(\lambda_1),\]
 \[  -\pi\leq \lambda_1<\lambda_2<\pi,\quad p,q=1,2,\dots,T.
\]
\end{theorem}

Consider another vector-valued stochastic sequence with the stationary GM
increments $\vec \zeta (m)=\vec \xi(m)+\vec \eta(m)$, where $\vec\eta(m)$ is a vector-valued  stationary stochastic sequence, uncorrelated with $\vec\xi(m)$, with the spectral representation
\[
 \vec\eta(m)=\int_{-\pi}^{\pi}e^{i\lambda m}dZ_{\eta}(\lambda),\]
 where $Z_{\eta}(\lambda)=\{Z_{\eta,p}(\lambda)\}_{p=1}^T$, $\lambda\in[-\pi,\pi)$, is a stochastic process with uncorrelated increments, that corresponds to the spectral function $G(\lambda)$ \cite{Hannan}.
The stochastic stationary GM increment $\chi_{\overline{\mu},\overline{s}}^{(d)}(\vec{\zeta}(m))$ allows the spectral representation
\begin{eqnarray*}
 \chi_{\overline{\mu},\overline{s}}^{(d)}(\vec{\zeta}(m))&=&\int_{-\pi}^{\pi}e^{i\lambda m}\frac{\chi_{\overline{\mu}}^{(d)}(e^{-i\lambda})}{\beta^{(d)}(i\lambda)}
 dZ_{\xi^{(n)}}(\lambda)
  +\int_{-\pi}^{\pi}e^{i\lambda m}\chi_{\overline{\mu}}^{(d)}(e^{-i\lambda}) dZ_{\eta }(\lambda),
 \end{eqnarray*}
while $dZ_{\eta }(\lambda)=(\beta^{(d)}(i\lambda))^{-1} dZ_{\eta^{(n)}}(\lambda)$,
$\lambda\in[-\pi,\pi)$. Therefore, in the case where the spectral functions $F(\lambda)$ and $G(\lambda)$ have the spectral densities $f(\lambda)$ and $g(\lambda)$, the spectral density $p(\lambda)=\{p_{ij}(\lambda)\}_{i,j=1}^{T}$ of the stochastic sequence $\vec \zeta(m)$ is determined by the formula
\[
 p(\lambda)=f(\lambda)+|\beta^{(d)}(i\lambda)|^2g(\lambda).\]

For a regular  stationary GM increment sequence
$\chi_{\overline{\mu},\overline{s}}^{(d)}(\vec{\xi}(m))$ \cite{Luz_Mokl_filt_GMI}, there exists an
innovation sequence ${\vec\varepsilon(u)=\{\varepsilon_k(u)\}_{k=1}^q, u \in\mathbb Z}$
and a sequence of matrix-valued
functions $\varphi^{(d)}(k,\overline{\mu}) =\{\varphi^{(d)}_{ij}(k,\overline{\mu}) \}_{i=\overline{1,T}}^{j=\overline{1,q}}$, $k\geq0$, such that
\begin{equation}\label{odnostRuhSer}
\sum_{k=0}^{\infty}
\sum_{i=1}^{T}
\sum_{j=1}^{q}
|\varphi^{(d)}_{ij}(k,\overline{\mu})|^2
<\infty,\quad
\chi_{\overline{\mu},\overline{s}}^{(d)}(\vec{\xi}(m))=
\sum_{k=0}^{\infty}\varphi^{(d)}(k,\overline{\mu})\vec\varepsilon (m-k).
\end{equation}
Representation (\ref{odnostRuhSer}) is called \emph{a canonical
moving average representation} of the stochastic stationary GM increment
sequence $\chi_{\overline{\mu},\overline{s}}^{(d)}(\vec{\xi}(m))$. Its  spectral function $F(\lambda)$
 has the spectral density  $f(\lambda)=\{f_{ij}(\lambda)\}_{i,j=1}^T$ admitting the canonical
factorization
\begin{equation*}\label{SpectrRozclad_f}
f(\lambda)=
\varphi(e^{-i\lambda})\varphi^*(e^{-i\lambda}),
\end{equation*}
where the function
$\varphi(z)=\sum_{k=0}^{\infty}\varphi(k)z^k$ has
analytic in the unit circle $\{z:|z|\leq1\}$
components
$\varphi_{ij}(z)=\sum_{k=0}^{\infty}\varphi_{ij}(k)z^k; i=1,\dots,T; j=1,\dots,q$.
Based on moving average  representation
$(\ref{odnostRuhSer})$ define
\[\varphi_{\overline{\mu}}(z)=\sum_{k=0}^{\infty} \varphi^{(d)}(k,\overline{\mu})z^k=\sum_{k=0}^{\infty}\varphi_{\overline{\mu}}(k)z^k.\]
 Then the following relation holds true:
\begin{equation}
\varphi_{\overline{\mu}}(e^{-i\lambda})
\varphi^*_{\overline{\mu}}(e^{-i\lambda})=
 \frac{|\chi_{\overline{\mu}}^{(d)}(e^{-i\lambda})|^2}{|\beta^{(d)}(i\lambda)|^2}f(\lambda)=
 \prod_{j=1}^r\frac{\ld|1-e^{-i\lambda\mu_js_j}\rd|^{2d_j}}{\prod_{k_j=-[s_j/2]}^{[s_j/2]}|\lambda-2\pi  k_j/s_j|^{2d_j}}f(\lambda).
\label{dd}
\end{equation}

The one-sided moving average representation (\ref{odnostRuhSer}) and
relation  (\ref{dd}) are used for finding the mean square optimal
estimates of unobserved values of vector-valued  sequences with stationary GM increments.

\section{Hilbert space projection method of filtering}\label{classical_filtering}

\subsection{Filtering of vector-valued stochastic sequence with stationary GM increments} \label{classical_filtering_vector}

Consider a vector-valued stochastic sequence $\vec{\xi}(m)$ with stationary GM increments constructed from transformation \eqref{PerehidXi} and a vector-valued stationary stochastic sequence $\vec\eta(m)$ uncorrelated with the sequence $\vec\xi(m)$.
Let the stationary GM increment sequence $\chi_{\overline{\mu},\overline{s}}^{(d)}(\vec{\xi}(m))=\{\chi_{\overline{\mu},\overline{s}}^{(d)}(\xi_p(m))\}_{p=1}^{T}$
and the stationary  sequence $\vec\eta(m)$ have  absolutely continuous spectral functions $F(\lambda)$ and $G(\lambda)$ with the spectral densities $f(\lambda)=\{f_{ij}(\lambda)\}_{i,j=1}^{T}$ and $g(\lambda)=\{g_{ij}(\lambda)\}_{i,j=1}^{T}$ respectively.
Without loss of generality   assume that $ \mathsf{E}\chi_{\overline{\mu},\overline{s}}^{(d)}(\vec{\xi}(m))=0$, $\mathsf{E}\vec\eta(m)=0$ and $\overline{\mu}>\overline{0}$.

\textbf{Filtering problem.} Consider the problem of mean square optimal linear estimation of the functional
\begin{equation}
A\vec{\xi}=\sum_{k=0}^{\infty}(\vec{a}(k))^{\top}\vec{\xi}(-k),
\end{equation}
which depends on unobserved values of a stochastic sequence $\vec{\xi}(k)=\{\xi_{p}(k)\}_{p=1}^{T}$ with stationary GM
increments. Estimates are based on observations of the sequence $\vec\zeta(k)=\vec\xi(k)+\vec\eta(k)$ at points $k=0,-1,-2,\ldots$.

We  suppose that the  conditions on coefficients $\vec{a}(k)=\{a_{p}(k)\}_{p=1}^{T}$, $k\geq0$
 \be\label{umova
na a_f_st.n_d}
\sum_{k=0}^{\infty}\|\vec{a}(k)\|<\infty,\quad
\sum_{k=0}^{\infty}(k+1)\|\vec{a}(k)\|^{2}<\infty,\ee
and  \emph{the minimality condition} on the spectral densities $f(\lambda)$ and $g(\lambda)$
\be
 \ip \text{Tr}\left[ \frac{|\beta^{(d)}(i\lambda)|^2}{|\chi_{\overline{\mu}}^{(d)}(e^{-i\lambda})|^2}\ld(f(\lambda)+|\beta^{(d)}(i\lambda)|^2 g(\lambda)\rd)^{-1}\right]
 d\lambda<\infty.
\label{umova11_f_st.n_d}
\ee
are satisfied.
The second condition \eqref{umova11_f_st.n_d} is the necessary and sufficient one under which the mean square error of the optimal estimate of functional  $A\vec\xi$ is not equal to $0$.

Any linear estimate $\widehat{A}\vec\xi$ of the functional $A\vec\xi$
allows the representation
\be
 \label{otsinka A_f_st.n_d}
 \widehat{A}\vec\xi=\sum_{k=0}^{\infty}(\vec a(k))^{\top}(\vec\xi(-k)+\vec\eta(-k))-\ip
 (\vec h_{\overline{\mu}}(\lambda))^{\top}dZ_{\xi^{(n)}+\eta^{(n)}}(\lambda), \ee
where
$ \vec h_{\overline{\mu}}(\lambda)=\{ h_p(\lambda)\}_{p=1}^T$ is the spectral characteristic of the estimate
$\widehat{A}\vec\eta$.

In the Hilbert space
$L_2(p)$, define
\[L_2^{0}(p)=\overline{span}\{
 e^{i\lambda k}\chi_{\overline{\mu}}^{(d)}(e^{-i\lambda})(\beta^{(d)}(i\lambda))^{-1}\vec\delta_l:\,\vec\delta_l=\{\delta_{lp}\}_{p=1}^T,\,l=1,\ldots,T,\,k\leq0\} .\]

Define the following matrix-valued Fourier coefficients:
\begin{eqnarray*}
 S_{\overline{\mu}}(k)&=&\frac{1}{2\pi}\ip e^{-i\lambda k}
 \frac{|\beta^{(d)}(i\lambda)|^2}{|\chi_{\overline{\mu}}^{(d)}(e^{-i\lambda})|^2}\ld[g(\lambda)(f(\lambda)+|\beta^{(d)}(i\lambda)|^2g(\lambda))^{-1}\rd]^{\top}d\lambda,
 \quad k\in \mr Z,
 \\
 P_{\overline{\mu}}(k)&=&\frac{1}{2\pi}\ip e^{-i\lambda k}
 \dfrac{|\beta^{(d)}(i\lambda)|^2}{|\chi_{\overline{\mu}}^{(d)}(e^{-i\lambda})|^2}\ld[(f(\lambda)+|\beta^{(d)}(i\lambda)|^2g(\lambda))^{-1}\rd]^{\top}
 d\lambda,\quad k\in \mr Z,
 \\
 Q(k)&=&\frac{1}{2\pi}\ip
 e^{-i\lambda k} \ld[f(\lambda)(f(\lambda)+|\beta^{(d)}(i\lambda)|^2g(\lambda))^{-1}g(\lambda)\rd]^{\top}d\lambda,\quad k\in \mr Z.
 \end{eqnarray*}
Define the vectors $\me a=((\vec a(0))^{\top},(\vec a(1))^{\top},(\vec a(2))^{\top},\ldots)^{\top}$ and  $ \me a_{\overline{\mu}}=((\vec a_{\overline{\mu}}(0))^{\top},
(\vec a_{\overline{\mu}}(1))^{\top},(\vec a_{\overline{\mu}}(2))^{\top},\ldots)^{\top}$, where the  coefficients
$ \vec a_{\overline{\mu}}(k)=\vec a_{-{\overline{\mu}}}(k-n(\gamma))$, $k\geq0$,
\be\label{coeff a_mu_f_st.n_d}
 \vec a_{-\overline{\mu}}(m)=\sum_{l=\max\ld\{m,0\rd\}}^{m+n(\gamma)}e_{\gamma}(l-m)\vec a(l),\quad m\geq-n(\gamma).\ee
Define the matrices $\me S_{\overline{\mu}}$, $\me P_{\overline{\mu}}$  and $\me Q$  by the matrix-valued entries $(\me S_{\overline{\mu}})_{l, k} =S_{\overline{\mu}}(l+1+k-n(\gamma))$, $(\me P_{\overline{\mu}})_{l,k}=P_{\overline{\mu}}(l-k)$ and  $(\me Q)_{l,k}=Q(l-k)$,  $l,k\geq0$.

The solution to the filtering problem   is described by the following theorem in terms of Fourier coefficients $\{S_{\overline{\mu}}(k),P_{\overline{\mu}}(k),Q(k):k\in\mr Z\}$ .

\begin{theorem}[\cite{Luz_Mokl_filt_GMI}]\label{thm_est_A}
A solution $\widehat{A}\vec\xi$ to the filtering problem for the linear functional $A\vec{\xi}$ of the values of a vector-valued stochastic sequence $\vec{\xi}(m)$ with
stationary  GM increments under conditions  (\ref{umova na a_f_st.n_d}) and (\ref{umova11_f_st.n_d}) is  calculated by formula (\ref{otsinka A_f_st.n_d}).
The spectral characteristic $\vec h_{\overline{\mu}}(\lambda)$ and the value of the mean square error $\Delta(f,g;\widehat{A}\vec\xi)$ are calculated by the formulas
\be\label{spectr A_f_st.n_d}
 (\vec h_{\overline{\mu}}(\lambda))^{\top}=\ld[\chi_{\overline{\mu}}^{(d)}(e^{i\lambda})(A(e^{-i\lambda }))^{\top}
  g(\lambda)-(C_{\overline{\mu}}(e^{i\lambda}))^{\top}\rd]p^{-1}(\lambda)\frac{\overline{\beta^{(d)}(i\lambda)}
}{\chi_{\overline{\mu}}^{(d)}(e^{i\lambda})},\ee
where
\[
A(e^{-i\lambda })=\sum_{k=0}^{\infty}\vec a(k)e^{-i\lambda k},
\quad
 C_{\overline{\mu}}(e^{i\lambda})=\sum_{k=0}^{\infty} \ld(\me
 P_{\overline{\mu}}^{-1}\me S_{\overline{\mu}}\me a_{\overline{\mu}}\rd)_ke^{i\lambda
 (k+1)},\]
 and
\begin{eqnarray}
  \Delta\ld(f,g;\widehat{A}\vec\xi\rd)=\mt E \ld|A\vec\xi-\widehat{A}\vec\xi\rd|^2=\ld\langle \me S_{\overline{\mu}}
 \me a_{\overline{\mu}},\me P_{\overline{\mu}}^{-1}\me S_{\overline{\mu}}\me a_{\overline{\mu}}\rd\rangle+\ld\langle\me Q\me a,\me
 a\rd\rangle.\label{poh A_f_st.n_d}\end{eqnarray}
\end{theorem}

\begin{remark}
The filtering problem in the presence of fractional integration is considered in \cite{Luz_Mokl_filt_GMI}.
\end{remark}

\subsection{Filtering based on factorizations of the spectral densities}

The main goal of the article is to derive the classical and minimax estimates of the functional $A\vec \xi$ in terms of the coefficients of the canonical factorizations of the spectral densities $f(\lambda)$, $g(\lambda)$ and $f(\lambda)+|\beta^{(d)}(i\lambda)|^2g(\lambda)$.

Let the following canonical factorizations take place
\begin{equation} \label{fakt1}
 \frac{|\beta^{(d)}(i\lambda)|^2}{|\chi_{\overline{\mu}}^{(d)}(e^{-i\lambda})|^2}
 (f(\lambda)+|\beta^{(d)}(i\lambda)|^2g(\lambda))
 =\Theta_{\overline{\mu}}(e^{-i\lambda})\Theta_{\overline{\mu}}^*(e^{-i\lambda}),
\quad \Theta_{\overline{\mu}}(e^{-i\lambda})=\sum_{k=0}^{\infty}\theta_{\overline{\mu}}(k)e^{-i\lambda k},
\end{equation}
\be \label{fakt3}
 g(\lambda)=\sum_{k=-\infty}^{\infty}g(k)e^{i\lambda k}=\Phi(e^{-i\lambda})\Phi^*(e^{-i\lambda}), \quad
 \Phi(e^{-i\lambda})=\sum_{k=0}^{\infty}\phi(k)e^{-i\lambda k}.
 \ee
Define the matrix-valued function $\Psi_{\overline{\mu}}(e^{-i\lambda})= \{\Psi_{\overline{\mu},ij}(e^{-i\lambda})\}_{i=\overline{1,q}}^{j=\overline{1,T}}$ by the equation
\[
\Psi_{\overline{\mu}}(e^{-i\lambda})\Theta_{\overline{\mu}}(e^{-i\lambda})=E_q,
\]
where $E_q$ is an identity $q\times q$ matrix.
One can check that the following factorization takes place
\begin{equation} \label{fakt2}
\frac{|\beta^{(d)}(i\lambda)|^2}{|\chi_{\overline{\mu}}^{(d)}(e^{-i\lambda})|^2}
 (f(\lambda)+|\beta^{(d)}(i\lambda)|^2g(\lambda))^{-1} =
 \Psi_{\overline{\mu}}^*(e^{-i\lambda})\Psi_{\overline{\mu}}(e^{-i\lambda}), \quad
 \Psi_{\overline{\mu}}(e^{-i\lambda})=\sum_{k=0}^{\infty}\psi_{\overline{\mu}}(k)e^{-i\lambda k},
\end{equation}

\begin{remark}\label{remark_density_adjoint}
Any spectral density matrix $f(\lambda)$ is self-adjoint: $f(\lambda)=f^*(\lambda)$. Thus, $(f(\lambda))^{\top}=\overline{f(\lambda)}$. One can check that an inverse spectral density $f^{-1}(\lambda)$ is also self-adjoint $f^{-1}(\lambda)=(f^{-1}(\lambda))^*$ and $(f^{-1}(\lambda))^{\top}=\overline{f^{-1}(\lambda)}$.
\end{remark}

The following Lemmas provide representations of  $\me P_{\overline{\mu}}$ and $\me P_{\overline{\mu}}^{-1}\me S_{\overline{\mu}}\me a_{\overline{\mu}}$, which contain  coefficients of factorizations (\ref{fakt1}) -- (\ref{fakt2}).

\begin{lemma}\label{lema_fact_2}
Let factorization (\ref{fakt1})  takes place and let $q\times T$ matrix function $\Psi_{\overline{\mu}}(e^{-i\lambda})$ satisfy an equation $\Psi_{\overline{\mu}}(e^{-i\lambda})\Phi_{\overline{\mu}}(e^{-i\lambda})=E_q$. Define the linear operators
 $ \Psi_{\overline{\mu}}$ and  $ \Theta_{\overline{\mu}}$ in the space $\ell_2$ by the matrices with the matrix entries
 $( \Psi_{\overline{\mu}})_{k,j}=\psi_{\overline{\mu}}(k-j)$, $( \Theta_{\overline{\mu}})_{k,j}=\theta_{\overline{\mu}}(k-j)$ for $0\leq j\leq k$, $(
\Psi_{\overline{\mu}})_{k,j}=0$, $(
\Theta_{\overline{\mu}})_{k,j}=0$ for $0\leq k<j$.
Then:
\\
a) the linear operator $\me P_{\overline{\mu}}$  admits the factorization \[\me
P_{\overline{\mu}}=(\Psi_{\overline{\mu}})^{\top} \overline{\Psi}_{\overline{\mu}};\]
\\
b) the inverse operator $(\me
P_{\overline{\mu}})^{-1}$ admits the factorization
\[
 (\me
P_{\overline{\mu}})^{-1}= \overline{\Theta}_{\overline{\mu}}(\Theta_{\overline{\mu}})^{\top}.\]
\end{lemma}
\begin{proof} See \cite{Luz_Mokl_extra_noise_PCI}.
\end{proof}

\begin{lemma}\label{lema_fact_3}
  Let  factorizations (\ref{fakt1}) and (\ref{fakt3}) take place.
 Define by $ \widetilde{e}_{\overline{\mu}}(m)=\ld(\Theta^{\top}_{\overline{\mu}}\me S_{\overline{\mu}}\widetilde{\me a}_{\overline{\mu}}\rd)_m$, $m\geq0$, the $m$th element of the vector
$\widetilde{\me e}_{\overline{\mu}}=\Theta^{\top}_{\overline{\mu}}\me S_{\overline{\mu}}\widetilde{\me a}_{\overline{\mu}}$. Then
\[
\widetilde{e}_{\overline{\mu}}(m)=\sum_{j=-n(\gamma)}^{\infty}Z_{\overline{\mu}}(m+j+1)a_{-\mu}(j),
\]
  where $Z_{\overline{\mu}}(j)$, $j\in \mr Z$, are defined as
\begin{eqnarray*}
  Z_{\overline{\mu}}(j)=\sum_{l=0}^{\infty}\overline{\psi}_{\overline{\mu}}(l)\overline{g}(l-j),\, j\in \mr Z,
 \quad
 g(k)=\sum_{m=\max\{0,-k\}}^{\infty}\phi(m)\phi^*(k+m),\, k\in\mr Z.
\end{eqnarray*}
\end{lemma}
\begin{proof}
See Appendix.
\end{proof}

 Define  the linear operators $\me G^{-}$, $\me G^{+}$,$\widetilde{\Phi}^{+}$ in the space $\ell_2$ by matrices with the matrix entries $(\me G^{-})_{l,k}=\overline{g}(l-k)$, $(\me G^{+})_{l,k}=\overline{g}(l+k)$, $\phi^{\top}(k+j)$ $l,k\in \mr Z$.  And the linear operator
  $\widetilde{\Phi}$ in the space $\ell_2$ determined by a matrix with the matrix entries
  $(\widetilde{\Phi})_{k,j}=\phi^{\top}(k-j)$ for $0\leq j\leq k$,  $ (\widetilde{\Phi})_{k,j}=0$ for $0\leq k<j$.

 Define also the coefficients $\{\vec b_{-\mu}(k):k\geq0\}$ as follows: $\vec b_{-\mu}(0)=0$, $\vec b_{-\mu}(k)=\vec a_{-\mu}(-k)$ for $1\leq k\leq n(\gamma)$, $\vec b_{-\mu}(k)=0$ for $k>n(\gamma)$, where coefficients $\vec a_{-\mu}(k)$ are calculated by formula (\ref{coeff a_mu_f_st.n_d}), and the vectors
\[\me a_{-\mu}=((\vec a_{-\mu}(0))^{\top},(\vec a_{-\mu}(1))^{\top},(\vec a_{-\mu}(2))^{\top},\ldots)^{\top}, \quad \me b_{-\mu}=((\vec b_{-\mu}(0))^{\top},(\vec b_{-\mu}(1))^{\top},(\vec b_{-\mu}(2))^{\top},\ldots)^{\top}.\]

The following theorem describes a solution to the filtering problem in the case when the spectral densities $f(\lambda)$ and $g(\lambda)$  admit  canonical factorizations (\ref{fakt1}) -- (\ref{fakt2}) .

\begin{theorem}\label{thm3_f_st.n_d_fact}
A solution $\widehat{A}\vec\xi$ to the filtering problem for the linear functional $A\vec{\xi}$ of the values of a vector-valued stochastic sequence $\vec{\xi}(m)$ with
stationary  GM increments under condition   (\ref{umova na a_f_st.n_d}) and provided that the spectral densities $f(\lambda)$ and $g(\lambda)$ of the stochastic vector sequences $\vec\xi(m)$ and $\vec\eta(m)$ admit  canonical factorizations (\ref{fakt1}) -- (\ref{fakt2}) is  calculated by formula (\ref{otsinka A_f_st.n_d}).
The spectral characteristic $\vec h_{\overline{\mu}}(\lambda)$ is calculated by the formulas
\begin{eqnarray}
\notag \vec h_{\overline{\mu}}(\lambda) &=&\frac{\chi_{\overline{\mu}}^{(d)}(e^{-i\lambda})}{\beta^{(d)}(i\lambda)}
\ld(\sum_{k=0}^{\infty}\psi^{\top}_{\overline{\mu}}(k)e^{-i\lambda k}\rd)
\sum_{m=0}^{\infty}( (\widetilde{\Psi}_{\overline{\mu}})^*\me G^{-}\me a_{-\mu} +(\widetilde{\Psi}_{\overline{\mu}})^* \me G^{+}\me b_{-\mu})_m e^{-i\lambda m}
\\
 &=&\frac{\chi_{\overline{\mu}}^{(d)}(e^{-i\lambda})}{\beta^{(d)}(i\lambda)}
\ld(\sum_{k=0}^{\infty}\psi^{\top}_{\overline{\mu}}(k)e^{-i\lambda k}\rd)
\sum_{m=0}^{\infty}(\overline{\psi}_{\overline{\mu}} \me C^{-}_{\overline{\mu},g} +\overline{\psi}_{\overline{\mu}} \me C^{+}_{\overline{\mu},g})_m e^{-i\lambda m},\label{spectr A_f_st.n_d_fact}
\end{eqnarray}
where $\overline{\psi}_{\overline{\mu}}=(\overline{\psi}_{\overline{\mu}}(0),\overline{\psi}_{\overline{\mu}}(1),\overline{\psi}_{\overline{\mu}}(2),\ldots)$,
\[
(\overline{\psi}_{\overline{\mu}} \me C^{\pm}_{\overline{\mu},g} )_m=\sum_{k=0}^{\infty}\overline{\psi}_{\overline{\mu}}(k)\me c^{\pm}_{\overline{\mu},g}(k+m),
\]
\[
\notag \me c^{-}_{\overline{\mu},g}(m)=\sum_{k=0}^{\infty}\overline{g}(m-k)a_{-\overline{\mu}}(k)
=\sum_{l=0}^{\infty}\overline{\phi}(l)\sum_{k=0}^{l+m}\phi^{\top}(l+m-k)a_{-\overline{\mu}}(k)
=\sum_{l=0}^{\infty}\overline{\phi}(l)(\widetilde{\Phi}\me a_{-\overline{\mu}})_{l+m}=(\widetilde{\Phi}^*\widetilde{\Phi}\me a_{-\overline{\mu}})_{m},
\]
\[
\me c^{+}_{\overline{\mu},g}(m)=\sum_{k=0}^{\infty}\overline{g}(m+k)b_{-\overline{\mu}}(k)
=\sum_{l=0}^{\infty}\overline{\phi}(l)\sum_{k=0}^{\infty}\phi^{\top}(l+m+k)b_{-\overline{\mu}}(k)
=\sum_{l=0}^{\infty}\overline{\phi}(l)(\widetilde{\Phi}^+\me b_{-\overline{\mu}})_{l+m}=(\widetilde{\Phi}^*\widetilde{\Phi}^+\me b_{-\overline{\mu}})_{m}.
\]
The  the value of the mean square error $\Delta(f,g;\widehat{A}\vec\xi)$ is calculated by the formulas
\begin{eqnarray}
 \Delta\ld(f,g;\widehat{A}\xi\rd)=\|\widetilde{\Phi}\me a\|^2-\|\overline{\psi}_{\overline{\mu}}( \me C^{-}_{\overline{\mu},g} + \me C^{+}_{\overline{\mu},g})\|^2.\label{poh A_f_st.n_d_fact}
\end{eqnarray}
\end{theorem}
\begin{proof} See Appendix.
\end{proof}

\begin{remark}
The following factorizations  hold true:
\[ \me G^{-}=\widetilde{\Phi}^{*}\widetilde{\Phi}, \quad \me G^{+}=\widetilde{\Phi}^{*}\widetilde{\Phi}^{+}.\]
\end{remark}

 The filtering problem  for the functional  $A_N\vec\xi$ is solved directly by Theorem \ref{thm3_f_st.n_d_fact}  by putting $\vec a(k)=\vec 0$ for
$k>N$. To solve the filtering problem for the $p$th coordinate of the single vector $\vec\xi (-N)$, we put $\vec a(N)=\vec\delta_p$, $\vec a(k)=\vec 0$ for $k\neq N$.

The following corollaries take place.

\begin{corollary}\label{nas_A_N_f_st.n_d_fact}
A solution $\widehat{A}_N\vec\xi$ to the filtering problem for the linear functional $A_N\vec{\xi}$ of the values of a vector-valued stochastic sequence $\vec{\xi}(m)$  with
stationary  GM increments under condition  (\ref{umova na a_f_st.n_d})  is  calculated by the formula
\be \label{otsinka A_N_f_st.n_d_fact}
 \widehat{A}_N\vec\xi=\sum_{k=0}^N(\vec a(k))^{\top}(\vec\xi(-k)+\vec\eta(-k))-\ip
 (\vec h_{\overline{\mu},N}(\lambda))^{\top}dZ_{\xi^{(n)}+\eta^{(n)}}(\lambda).
\ee
The spectral characteristic
$\vec h_{\overline{\mu},N}(\lambda)$ and the  value of the mean square error $\Delta(f,g;\widehat{A}_N\vec\xi)$ of the optimal estimate $\widehat{A}_N\vec\xi$ are calculated by   formulas \eqref{spectr A_f_st.n_d_fact} and \eqref{poh A_f_st.n_d_fact} for the vectors $\me a$, $\me a_{-\mu}$, $\me b_{-\mu}$ calculated as
 \[\me a_N=((\vec a(0))^{\top},(\vec a(1))^{\top},(\vec a(2))^{\top},\ldots,(\vec a(N))^{\top},0,\ldots)^{\top},
\]
\[
\me a_{-\mu,N}=((\vec a_{-\mu,N}(0))^{\top},(\vec a_{-\mu,N}(1))^{\top},...,(\vec a_{-\mu,N}(N))^{\top},0,\ldots)^{\top},
\]
\[
\me b_{-\mu,N}=(0,(\vec a_{-\mu,N}(-1))^{\top},(\vec a_{-\mu,N}(-2))^{\top},...,(\vec a_{-\mu,N}(-n(\gamma)))^{\top},0 \ldots)^{\top}
\]
where
\begin{eqnarray}
 \vec a_{-\overline{\mu},N}(m)=\sum_{l=\max\ld\{m,0\rd\}}^{\min\{m+n(\gamma),N\}}e_{\gamma}(l-m)\vec a(l),\quad  -n(\gamma)\leq m\leq N. \label{coeff a_N_mu_f_st.n_d}
\end{eqnarray}
\end{corollary}

\begin{corollary} \label{nas xi_f_st.n_d}
The optimal linear estimate $\widehat{\xi}_p(-N)$ of an unobserved value
$ \xi_p(-N)$, $N\geq0$, of the stochastic vector sequence $\vec\xi(m)$ with GM stationary increments based on observations of the sequence $\vec\xi(m)+\vec\eta(m)$ at points $m=0,-1,-2,\ldots$, where the noise sequence $\vec\eta(m)$ is uncorrelated with $\vec\xi(m)$, is calculated by the formula
\be\label{otsinka A_p_f_st.n_d_fact}
 \widehat{\xi}_p(-N)=(\xi_p(-N)+\eta_p(-N))-\ip
 (\vec h_{\overline{\mu},N,p}(\lambda))^{\top}dZ_{\xi^{(n)}+\eta^{(n)}}(\lambda).\ee
Put
 \begin{eqnarray*}
 \vec a_{-\overline{\mu},N,p}(m)=e_{\gamma}(N-m)\vec \delta_p,\quad  N-n(\gamma)\leq m\leq N.
 \end{eqnarray*}
If $N<n(\gamma)$, the spectral characteristic
$\vec h_{\overline{\mu},N,p}(\lambda)$ and the  value of the mean square error $\Delta(f,g;\widehat{\xi}_p(-N))$ of the optimal estimate $\widehat{\xi}_p(-N)$ are calculated by   formulas \eqref{spectr A_f_st.n_d_fact} and \eqref{poh A_f_st.n_d_fact} for the vectors $\me a$, $\me a_{-\mu}$, $\me b_{-\mu}$ calculated as
 \[\me a_{N,p}=(0,0,\ldots,(\me \delta_p)^{\top},0,\ldots)^{\top},
\]
\[
\me a_{-\mu,N,p}=((\vec a_{-\mu,N,p}(0))^{\top},(\vec a_{-\mu,N,p}(1))^{\top},...,(\vec a_{-\mu,N,p}(N))^{\top},0,\ldots)^{\top},
\]
\[
\me b_{-\mu,N,p}=(0,(\vec a_{-\mu,N,p}(-1))^{\top},(\vec a_{-\mu,N,p}(-2))^{\top},...,(\vec a_{-\mu,N,p}(N-n(\gamma)))^{\top},0 \ldots)^{\top}.
\]
If $N\geq n(\gamma)$, the spectral characteristic
$\vec h_{\overline{\mu},N,p}(\lambda)$ and the  value of the mean square error $\Delta(f,g;\widehat{\xi}_p(-N))$ of the optimal estimate $\widehat{\xi}_p(-N)$ are calculated by   formulas
\begin{eqnarray}
\vec h_{\overline{\mu},N,p}(\lambda)=\frac{\chi_{\overline{\mu}}^{(d)}(e^{-i\lambda})}{\beta^{(d)}(i\lambda)}
\ld(\sum_{k=0}^{\infty}\psi^{\top}_{\overline{\mu}}(k)e^{-i\lambda k}\rd)
\sum_{m=0}^{\infty}( (\widetilde{\Psi}_{\overline{\mu}})^*\me G^{-}\me a_{-\mu,N,p} )_m e^{-i\lambda m}
\label{spectr A_p_f_st.n_d_2_fact}
\end{eqnarray}
and
\begin{eqnarray}
\notag \Delta\ld(f,g;\widehat{\xi}_p(-N)\rd)=\ld\langle\overline{g}(0)\vec\delta_p,\vec\delta_p\rd\rangle-\|(\widetilde{\Psi}_{\overline{\mu}})^*\me G^{-}\me a_{-\mu,N,p}\|^2,
\label{poh A_p_f_st.n_d_2_fact}
\end{eqnarray}
where
\[
\me a_{-\mu,N,p}=(0,\ldots,0,(\vec a_{-\mu,N,p}(N-n(\gamma)))^{\top},...,(\vec a_{-\mu,N,p}(N))^{\top},0,\ldots)^{\top}.
\]
\end{corollary}

\subsection{Filtering of stochastic sequences with periodically stationary GM increment}

Consider the filtering problem for the functionals
\begin{equation}
A{\xi}=\sum_{k=0}^{\infty}{a}^{(\xi)}(k)\xi(-k), \quad
A_{M}{\xi}=\sum_{k=0}^{N}{a}^{(\xi)}(k)\xi(-k)
\end{equation}
which depend on unobserved values of a stochastic sequence $\xi(m)$ with periodically stationary
GM increments. Estimates are based on observations of the sequence $\xi(m)+\eta(m)$ at points $m=0,-1,-2,\ldots$, where the periodically stationary noise sequence $ \eta(m)$ is uncorrelated with $ \xi(m)$.

The functional $A{\xi}$ can be represented in the form
\begin{eqnarray}
\nonumber
A{\xi}& = &\sum_{k=0}^{\infty}{a}^{(\xi)}(k)\xi(-k)=\sum_{m=0}^{\infty}\sum_{p=1}^{T}
{a}^{(\xi)}(mT+p-1)\xi(-mT-p+1)
\\\nonumber
& = & \sum_{m=0}^{\infty}\sum_{p=1}^{T}a_p(m)\xi_p(-m)=\sum_{m=0}^{\infty}(\vec{a}(m))^{\top}\vec{\xi}(-m)
=A\vec{\xi},
\end{eqnarray}
where
\be \label{zeta1}
\vec{\xi}(m)=({\xi}_1(m),{\xi}_2(m),\dots,{\xi}_T(m))^{\top},\,
 {\xi}_p(m)=\xi(mT+p-1);\,p=1,2,\dots,T;
\ee
and
\be \label{azeta}
 \vec{a}(m) =({a}_1(m),{a}_2(m),\dots,{a}_T(m))^{\top},\,
 {a}_p(m)=a^{(\xi)}(mT+p-1);\,p=1,2,\dots,T.
\ee

In the same way, the functional $A{\eta}$ is represented as
\[
A{\eta}=\sum_{k=0}^{\infty}{a}^{(\xi)}(k)\eta(-k)=\sum_{m=0}^{\infty}(\vec{a}(m))^{\top}\vec{\eta}(-m)
=A\vec{\eta},
\]
where
\be \label{zeta2}
\vec{\eta}(m)=({\eta}_1(m),{\eta}_2(m),\dots,{\eta}_T(m))^{\top},\,
 {\eta}_p(m)=\eta(mT+p-1);\,p=1,2,\dots,T.
\ee

Making use of the introduced notations and statements of Theorem \ref{thm_est_A} we can claim that the following theorem holds true.

\begin{theorem}
\label{thm_est_Azeta}
Let a stochastic sequence ${\xi}(m)$ with periodically stationary GM increments and a stochastic periodically stationary sequence ${\eta}(m)$   generate by formulas \eqref{zeta1} and \eqref{zeta2}
  vector-valued stochastic sequences $\vec{\xi}(m) $ and $\vec{\eta}(m) $ with  the spectral densities matrices $f(\lambda)=\{f_{ij}(\lambda)\}_{i,j=1}^{T}$ and $g(\lambda)=\{g_{ij}(\lambda)\}_{i,j=1}^{T}$  admitting  canonical factorizations (\ref{fakt1}) -- (\ref{fakt2}). A solution $\widehat{A} \xi$ to the filtering problem for the functional $A\xi=A\vec \xi$ under conditions  (\ref{umova na a_f_st.n_d})  is  calculated by formula
 (\ref{otsinka A_f_st.n_d}) for the coefficients $\vec a(m)$, $m\geq0$, defined in \eqref{azeta}.
The spectral characteristic
$ \vec h_{\overline{\mu}}(\lambda)=\{h_{p}(\lambda)\}_{p=1}^{T}$ and the value of the mean square error $\Delta(f;\widehat{A}\xi)$ of the   estimate $\widehat{A}\xi$ are calculated by formulas
(\ref{spectr A_f_st.n_d_fact}) and (\ref{poh A_f_st.n_d_fact}) respectively.
\end{theorem}

The functional $A_M{\xi}$ can be represented in the form
\begin{eqnarray}
\nonumber
A_M{\xi}& = &\sum_{k=0}^{M}{a}^{(\xi)}(k)\zeta(-k)=\sum_{m=0}^{N}\sum_{p=1}^{T}
{a}^{(\xi)}(mT+p-1)\xi(-mT-p+1)
\\\nonumber
& = &\sum_{m=0}^{N}\sum_{p=1}^{T}a_p(m)\xi_p(-m)=\sum_{m=0}^{N}(\vec{a}(m))^{\top}\vec{\xi}(-m)=A_N\vec{\xi},
\end{eqnarray}
where $N=[\frac{M}{T}]$, the sequence   $\vec{\xi}(m) $ is determined by formula \eqref{zeta1},
\begin{eqnarray}
\nonumber
(\vec{a}(m))^{\top}& = &({a}_1(m),{a}_2(m),\dots,{a}_T(m))^{\top},
\\\nonumber
 {a}_p(m)& = &a^{\zeta}(mT+p-1);\,0\leq m\leq N; 1\leq p\leq T;\,mT+p-1\leq M;
\\  {a}_p(N)& = &0;\quad
M+1\leq NT+p-1\leq (N+1)T-1;1\leq p\leq T. \label{aNzeta}
\end{eqnarray}

An estimate of a single unobserved value
$\xi(-M)$, $M\geq0$ of a stochastic sequence ${\xi}(m)$ with periodically stationary GM increments
is obtained by making use of the notations
$\xi(-M)=\xi_p(-N)=(\vec \delta_p)^{\top}\vec\xi(N)$, $N=[\frac{M}{T}]$, $p=M+1-NT$. We can conclude that the following corollaries hold true.

\begin{corollary}
\label{nas_est_A_Nzeta}
Let a stochastic sequence ${\xi}(m)$ with periodically stationary GM increments and a stochastic periodically stationary sequence ${\eta}(m)$   generate by formulas \eqref{zeta1} and \eqref{zeta2}
 vector-valued stochastic sequences $\vec{\xi}(m) $ and $\vec{\eta}(m) $. A solution $\widehat{A}_M\xi$ to the filtering problem for  the functional $A_M\xi=A_N\vec{\xi}$  under condition (\ref{umova11_f_st.n_d}) is  calculated by formula
 (\ref{otsinka A_N_f_st.n_d_fact}) for the coefficients $\vec a(m)$, $0\leq m \leq N$, defined in \eqref{aNzeta}.
The spectral characteristic and the value of the mean square error of the estimate $\widehat{A}_M\xi$
are calculated by formulas  \eqref{spectr A_f_st.n_d_fact} and \eqref{poh A_f_st.n_d_fact} respectively.
\end{corollary}

\begin{corollary}\label{nas zeta_e_d}
Let a stochastic sequence ${\xi}(m)$ with periodically stationary GM increments and a stochastic periodically stationary sequence ${\eta}(m)$   generate by formulas \eqref{zeta1} and \eqref{zeta2}
vector-valued stochastic sequences $\vec{\xi}(m) $ and $\vec{\eta}(m) $. A solution $\widehat{\xi}(-M)$ to the filtering problem for an unobserved value $\xi(-M)=\xi_p(-N)=(\vec \delta_p)^{\top}\vec\xi(-N)$, $N=[\frac{M}{T}]$, $p=M+1-NT$,  under condition (\ref{umova11_f_st.n_d})
 is calculated by formula \eqref{otsinka A_p_f_st.n_d_fact}.
The spectral characteristic and the value of the mean square error of the estimate $\widehat{\xi}(-M)$ are calculated by   formulas \eqref{spectr A_f_st.n_d_fact} and \eqref{poh A_f_st.n_d_fact} or
 \eqref{spectr A_p_f_st.n_d_2_fact} and \eqref{poh A_p_f_st.n_d_2_fact} respectively.
\end{corollary}

\section{Minimax (robust) method of filtering}\label{minimax_filtering}

Solutions of the problem of  estimating the functionals ${A}\vec\xi$ and ${A}_N\vec\xi$ constructed from unobserved values of the stochastic sequence $\vec{\xi}(m)$ with  stationary  GM increments
$\chi_{\overline{\mu},\overline{s}}^{(d)}(\vec{\xi}(m))$ having the spectral density matrix $f(\lambda)$
based on its observations with stationary  noise
$\vec\xi(m)+\vec\eta(m)$ at points $m=0,-1,-2,\dots$ are  proposed in Theorem \ref{thm3_f_st.n_d_fact} and Corollary \ref{nas_A_N_f_st.n_d_fact}  in the case where the spectral density matrices
$f(\lambda)$ and $g(\lambda)$ of the target sequence  and the noise are exactly known.

In this section, we study the case where the  complete information about the spectral density matrices is not available, while some  sets   of admissible spectral densities $\md D=\md D_f\times\md D_g$ is known.
The minimax approach of estimation of the functionals from  unobserved values of stochastic sequences is considered, which
consists in finding an estimate that minimizes
the maximal values of the mean square errors for all spectral densities
from a class $\md D$ simultaneously. This method will be applied for the concrete classes of spectral densities.

The proceed with the stated problem, we recall the following definitions \cite{Moklyachuk}.

\begin{defn}
For a given class of spectral densities $\mathcal{D}=\md
D_f\times\md D_g$, the spectral densities
$f^0(\lambda)\in\mathcal{D}_f$, $g^0(\lambda)\in\md D_g$
are called the least favourable densities in the class $\mathcal{D}$ for
optimal linear filtering of the functional $A\xi$ if the following relation holds true
\[
 \Delta(f^0,g^0)=\Delta(h(f^0,g^0);f^0,g^0)=
 \max_{(f,g)\in\mathcal{D}_f\times\md
 D_g}\Delta(h(f,g);f,g).\]
\end{defn}

\begin{defn}
For a given class of spectral
densities $\mathcal{D}=\md D_f\times\md D_g$ the spectral
characteristic $\vec h^0(\lambda)$ of the optimal estimate of the functional
$A\xi$ is called minimax (robust) if the following relations hold true
\begin{eqnarray*}
 \vec h^0(\lambda)\in H_{\mathcal{D}}
 &=&\bigcap_{(f,g)\in\mathcal{D}_f\times\md D_g}L_2^{0}(p),
 \\
 \min_{\vec h\in H_{\mathcal{D}}}\max_{(f,g)\in \mathcal{D}_f\times\md D_g}\Delta(\vec h;f,g)
 &=&\max_{(f,g)\in\mathcal{D}_f\times\md D_g}\Delta(\vec h^0;f,g).
 \end{eqnarray*}
\end{defn}

Taking into account the introduced definitions and the relations derived in the previous sections  we can verify that the following lemma holds true.

\begin{lemma} The spectral densities $f^0\in\mathcal{D}_f$,
$g^0\in\mathcal{D}_g$ which admit  canonical factorizations (\ref{dd}), (\ref{fakt1}) and (\ref{fakt3})
are least favourable densities in the class $\mathcal{D}$ for the optimal linear filtering
of the functional $A\vec \xi$ based on observations of the sequence $\vec \xi(m)+\vec \eta(m)$ at points $m=0,-1,-2,\ldots$ if the matrix coefficients
of  canonical factorizations
(\ref{fakt1}) and (\ref{fakt3})
determine a solution to the constrained optimization problem
\begin{eqnarray}
\|\widetilde{\Phi}\me a\|^2-\|\overline{\psi}_{\overline{\mu}}( \me C^{-}_{\overline{\mu},g} + \me C^{+}_{\overline{\mu},g})\|^2\rightarrow\sup, \label{minimax1_f_st.n_d_fact}
\end{eqnarray}
\begin{eqnarray*}
 f(\lambda)&=&\frac{|\beta^{(d)}(i\lambda)|^2}{|\chi_{\overline{\mu}}^{(d)}(e^{-i\lambda})|^2}
\Theta_{\overline{\mu}}(e^{-i\lambda})\Theta_{\overline{\mu}}^*(e^{-i\lambda})
-|\beta^{(d)}(i\lambda)|^2\Phi(e^{-i\lambda})\Phi^*(e^{-i\lambda})\in \mathcal{D}_f,
\\
 g(\lambda)&=&\Phi(e^{-i\lambda})\Phi^*(e^{-i\lambda})\in \mathcal{D}_g.
\end{eqnarray*}
The minimax spectral characteristic $\vec h^0=\vec h_{\overline{\mu}}(f^0,g^0)$ is calculated by formula (\ref{spectr A_f_st.n_d_fact}) if
$\vec h_{\overline{\mu}}(f^0,g^0)\in H_{\mathcal{D}}$.
\end{lemma}

\begin{lemma} The spectral density $g^0\in\mathcal{D}_g$ which admits  canonical factorizations (\ref{fakt1}), (\ref{fakt3}) with the known spectral density $f(\lambda)$ is the least favourable in the class $\mathcal{D}_g$ for the optimal linear filtering
of the functional $A\xi$ based on observations of the sequence $\vec \xi(m)+\vec \eta(m)$ at points $m=0,-1,-2,\ldots$ if the matrix coefficients
of the canonical factorizations
\begin{eqnarray*} \label{fakt24_1_lf}
 f(\lambda)+|\beta^{(d)}(i\lambda)|^2g^0(\lambda)&=&\frac{|\beta^{(d)}(i\lambda)|^2}{|\chi_{\overline{\mu}}^{(d)}(e^{-i\lambda})|^2}
\ld(\sum_{k=0}^{\infty}\theta^0_{\overline{\mu}}(k)e^{-i\lambda k}\rd)\ld(\sum_{k=0}^{\infty}\theta^0_{\overline{\mu}}(k)e^{-i\lambda k}\rd)^*,
\\ \label{fakt24_2_lf}
 g^0(\lambda)&=&\ld(\sum_{k=0}^{\infty}\phi^0(k)e^{-i\lambda k}\rd)\ld(\sum_{k=0}^{\infty}\phi^0(k)e^{-i\lambda k}\rd)^*
 \end{eqnarray*}
are determined by the equation $\Psi^0_{\overline{\mu}}(e^{-i\lambda})\Theta^0_{\overline{\mu}}(e^{-i\lambda})=E_q$  and a solution $\{\psi^0_{\overline{\mu}}(k),\phi^0(k):k\geq 0\}$ to the constrained optimization problem
\begin{eqnarray}
\|\widetilde{\Phi}\me a\|^2-\|\overline{\psi}_{\overline{\mu}}( \me C^{-}_{\overline{\mu},g} + \me C^{+}_{\overline{\mu},g})\|^2\rightarrow\sup,
 \label{minimax2_f_st.n_d_fact}
 \end{eqnarray}
\[
 g(\lambda)=\Phi(e^{-i\lambda})\Phi^*(e^{-i\lambda})\in \mathcal{D}_g.\]
The minimax spectral characteristic $\vec h^0=\vec h_{\overline{\mu}}(f,g^0)$ is calculated by formula (\ref{spectr A_f_st.n_d_fact}) if
$\vec h_{\overline{\mu}}(f,g^0)\in H_{\mathcal{D}}$.
\end{lemma}
\begin{lemma} The spectral density $f^0\in\mathcal{D}_f$ which admits  canonical factorizations (\ref{dd}), (\ref{fakt1}) with the known spectral density $g(\lambda)$ is the least favourable spectral density in the class
 $\md D_f$ for the optimal linear filtering
of the functional $A\vec \xi$ based on observations of the sequence $\vec \xi(m)+\vec \eta(m)$ at points $m=0,-1,-2,\ldots$ if matrix coefficients
of the canonical factorization
\begin{eqnarray*} \label{fakt2_lf}
f^0(\lambda)+|\beta^{(d)}(i\lambda)|^2g(\lambda)=\frac{|\beta^{(d)}(i\lambda)|^2}{|\chi_{\overline{\mu}}^{(d)}(e^{-i\lambda})|^2}
\ld(\sum_{k=0}^{\infty}\theta^0_{\overline{\mu}}(k)e^{-i\lambda k}\rd)\ld(\sum_{k=0}^{\infty}\theta^0_{\overline{\mu}}(k)e^{-i\lambda k}\rd)^*
\end{eqnarray*}
are determined by  the equation $\Psi^0_{\overline{\mu}}(e^{-i\lambda})\Theta^0_{\overline{\mu}}(e^{-i\lambda})=E_q$ and  a solution $\{\psi^0_{\overline{\mu}}(k):k\geq 0\}$ to the constrained optimization problem
\begin{eqnarray}
\|\overline{\psi}_{\overline{\mu}}( \me C^{-}_{\overline{\mu},g} + \me C^{+}_{\overline{\mu},g})\|^2\rightarrow\inf,
 \label{minimax3_f_st.n_d_fact}\end{eqnarray}
\begin{eqnarray*}
 f(\lambda)&=&\frac{|\beta^{(d)}(i\lambda)|^2}{|\chi_{\overline{\mu}}^{(d)}(e^{-i\lambda})|^2}
\Theta_{\overline{\mu}}(e^{-i\lambda})\Theta_{\overline{\mu}}^*(e^{-i\lambda})
-|\beta^{(d)}(i\lambda)|^2\Phi(e^{-i\lambda})\Phi^*(e^{-i\lambda})\in \mathcal{D}_f
\end{eqnarray*}
for the fixed matrix coefficients $\{\phi(k):k\geq0\}$. The minimax spectral characteristic $\vec h^0=\vec h_{\overline{\mu}}(f^0,g)$ is calculated by formula (\ref{spectr A_f_st.n_d_fact}) if
$\vec h_{\overline{\mu}}(f^0,g)\in H_{\mathcal{D}}$.
\end{lemma}

The more detailed analysis of properties of the least favorable spectral densities and the minimax-robust spectral characteristics shows that the minimax spectral characteristic $h^0$ and the least favourable spectral densities $f^0$ and $g^0$ form a saddle
point of the function $\Delta(h;f,g)$ on the set
$H_{\mathcal{D}}\times\mathcal{D}$.
The saddle point inequalities
\[
 \Delta(h;f^0,g^0)\geq\Delta(h^0;f^0,g^0)\geq\Delta(h^0;f,g)\quad\forall (f,g)\in
 \mathcal{D},\forall h\in H_{\mathcal{D}}\]
hold true if $h^0=h_{\overline{\mu}}(f^0,g^0)$,
$h_{\overline{\mu}}(f^0,g^0)\in H_{\mathcal{D}}$ and $(f^0,g^0)$ is a solution of the constrained optimization problem
\be
 \widetilde{\Delta}(f,g)=-\Delta(h_{\overline{\mu}}(f^0,g^0);f,g)\to
 \inf,\quad (f,g)\in \mathcal{D},\label{zad_um_extr_e_d}
 \ee
where the functional $\Delta(h_{\overline{\mu}}(f^0,g^0);f,g)$ is calculated by the formula
\begin{eqnarray*}
\Delta\ld(h_{\overline{\mu}}(f^0,g^0);f,g\rd)&=&\frac{1}{2\pi}\ip\frac{|\chi_{\overline{\mu}}^{(d)}(e^{-i\lambda})|^2}{|\beta^{(d)}(i\lambda)|^2}
(\me h^0_{\overline{\mu},f}(e^{-i\lambda}))^{\top}\Psi^0_{\overline{\mu}}(e^{-i\lambda })f(\lambda)
(\Psi^0_{\overline{\mu}}(e^{-i\lambda }))^*\overline{\me h^0_{\overline{\mu},f}(e^{-i\lambda})}d\lambda
\\
&&+
\frac{1}{2\pi}\ip
(\me h^0_{\overline{\mu},g}(e^{-i\lambda}))^{\top}\Psi^0_{\overline{\mu}}(e^{-i\lambda })g(\lambda)
(\Psi^0_{\overline{\mu}}(e^{-i\lambda }))^*\overline{\me h^0_{\overline{\mu},g}(e^{-i\lambda})}d\lambda,
\end{eqnarray*}
where
\begin{eqnarray*}
 \me h^0_{\overline{\mu},f}(e^{-i\lambda})&=&\sum_{m=0}^{\infty}(\overline{\psi}_{\overline{\mu}} \me C^{-}_{\overline{\mu},g} +\overline{\psi}_{\overline{\mu}} \me C^{+}_{\overline{\mu},g})_m e^{-i\lambda m},
\\
 \me h^0_{\overline{\mu},g}(e^{-i\lambda})&=&\ld(\sum_{k=0}^{\infty}\theta^{\top}_{\overline{\mu}}(k)e^{-i\lambda k}\rd)\ld(\sum_{k=0}^{\infty}a(k)e^{-i\lambda k}\rd)-\chi_{\overline{\mu}}^{(d)}(e^{-i\lambda})\ld(\sum_{m=0}^{\infty}(\overline{\psi}_{\overline{\mu}} \me C^{-}_{\overline{\mu},g} +\overline{\psi}_{\overline{\mu}} \me C^{+}_{\overline{\mu},g})_m e^{-i\lambda m}\rd).
\end{eqnarray*}

The constrained optimization problem (\ref{zad_um_extr_e_d}) is equivalent to the unconstrained optimization problem
\be \label{zad_unconst_extr_f_st_d}
 \Delta_{\mathcal{D}}(f,g)=\widetilde{\Delta}(f,g)+ \delta(f,g|\mathcal{D})\to\inf,\ee
where $\delta(f,g|\mathcal{D})$ is the indicator function of the set
$\mathcal{D}$, namely $\delta(f,g|\mathcal{D})=0$ if $(f,g)\in \mathcal{D}$ and $\delta(f|\mathcal{D})=+\infty$ if $(f,g)\notin \mathcal{D}$.
The condition
 $0\in\partial\Delta_{\mathcal{D}}(f^0,g^0)$ characterizes a solution $(f^0,g^0)$ of the stated unconstrained optimization problem. This condition is the necessary and sufficient condition that the point $(f^0,g^0)$ belongs to the set of minimums of the convex functional $\Delta_{\mathcal{D}}(f,g)$ \cite{Moklyachuk2015,Rockafellar}.
 Thus, it allows us to find the equalities for the least favourable spectral densities in some special classes of spectral densities $\md D$.

The form of the functional $\Delta(h_{\overline{\mu}}(f^0,g^0);f,g)$ is suitable for application of the Lagrange method of indefinite
multipliers to the constrained optimization problem \eqref{zad_um_extr_e_d}.
Thus, the complexity of the problem is reduced to  finding the subdifferential of the indicator function of the set of admissible spectral densities. We illustrate the solving of the problem \eqref{zad_unconst_extr_f_st_d} for concrete sets admissible spectral densities  in the following subsections. A semi-uncertain filtering problem, when the spectral density $f(\lambda)$ is known and the spectral density $g(\lambda)$ belongs to  in class  $\md D_g$, is considered as well.

\subsection{Least favorable spectral density in classes $\md D_0 \times \md D_{1\delta}$}
\label{class_D0}

Consider the minimax filtering problem for the functional $A\vec{\xi}$
  for sets of admissible spectral densities $\md D_0^k$, $k=1,2,3,4$ of the sequence with GM increments $\vec \xi(m)$
\begin{eqnarray*}
\md D_{0}^{1} &=&\bigg\{f(\lambda )\left|\frac{1}{2\pi} \int
_{-\pi}^{\pi}
\frac{|\chi_{\overline{\mu}}^{(d)}(e^{-i\lambda})|^2}{|\beta^{(d)}(i\lambda)|^2}
f(\lambda )d\lambda  =P\right.\bigg\},
\\
\md D_{0}^{2} &=&\bigg\{f(\lambda )\left|\frac{1}{2\pi }
\int _{-\pi }^{\pi}
\frac{|\chi_{\overline{\mu}}^{(d)}(e^{-i\lambda})|^2}{|\beta^{(d)}(i\lambda)|^2}
{\rm{Tr}}\,[ f(\lambda )]d\lambda =p\right.\bigg\},
\\
\md D_{0}^{3} &=&\bigg\{f(\lambda )\left|\frac{1}{2\pi }
\int _{-\pi}^{\pi}
\frac{|\chi_{\overline{\mu}}^{(d)}(e^{-i\lambda})|^2}{|\beta^{(d)}(i\lambda)|^2}
f_{kk} (\lambda )d\lambda =p_{k}, k=\overline{1,T}\right.\bigg\},
\\
\md D_{0}^{4} &=&\bigg\{f(\lambda )\left|\frac{1}{2\pi} \int _{-\pi}^{\pi}
\frac{|\chi_{\overline{\mu}}^{(d)}(e^{-i\lambda})|^2}{|\beta^{(d)}(i\lambda)|^2}
\left\langle B_{1} ,f(\lambda )\right\rangle d\lambda  =p\right.\bigg\},
\end{eqnarray*}
where  $p, p_k, k=\overline{1,T}$ are given numbers, $P, B_1$ are given positive-definite Hermitian matrices, and sets of admissible spectral densities $\md D_{V}^{U}$, $k=1,2,3,4$ for the stationary noise sequence $\vec \eta(m)$
\begin{eqnarray*}
\md D_{1\delta}^{1}&=&\left\{g(\lambda )\biggl|\frac{1}{2\pi} \int_{-\pi}^{\pi}
\left|g_{ij} (\lambda )-g_{ij}^{1} (\lambda)\right|d\lambda  \le \delta_{i}^j, i,j=\overline{1,T}\right\}.
\\
\md D_{1\delta}^{2}&=&\left\{g(\lambda )\biggl|\frac{1}{2\pi} \int_{-\pi}^{\pi}
\left|{\rm{Tr}}(g(\lambda )-g_{1} (\lambda))\right|d\lambda \le \delta\right\};
\\
\md D_{1\delta}^{3}&=&\left\{g(\lambda )\biggl|\frac{1}{2\pi } \int_{-\pi}^{\pi}
\left|g_{kk} (\lambda )-g_{kk}^{1} (\lambda)\right|d\lambda  \le \delta_{k}, k=\overline{1,T}\right\};
\\
\md D_{1\delta}^{4}&=&\left\{g(\lambda )\biggl|\frac{1}{2\pi } \int_{-\pi}^{\pi}
\left|\left\langle B_{2} ,g(\lambda )-g_{1}(\lambda )\right\rangle \right|d\lambda  \le \delta\right\};
\end{eqnarray*}
where  $g_{1} ( \lambda )=\{g_{ij}^{1} ( \lambda )\}_{i,j=1}^T$ is a fixed spectral density,  $B_2$ is a given positive-definite Hermitian matrix,
$\delta,\delta_{k},k=\overline{1,T}$, $\delta_{i}^{j}, i,j=\overline{1,T}$, are given numbers.

The condition $0\in\partial\Delta_{\mathcal{D}}(f^0,g^0)$
implies the following equations which determine the least favourable spectral densities for these given sets of admissible spectral densities.

For the first set of admissible spectral densities $\md D_{f0}^1 \times\md D_{g1\delta}^{1}$:
\begin{eqnarray} \label{eq_4_1f_fact}
\left(
{\me h^{0}_{\overline{\mu},f}(e^{i\lambda})}
\right)
\left(
{\me h^{0}_{\overline{\mu},f}(e^{i\lambda})}
\right)^{*}
&=&(\Theta_{\overline{\mu}}(e^{-i\lambda}))^{\top}
\vec{\alpha}_f\cdot \vec{\alpha}_f^{*}
\overline{\Theta_{\overline{\mu}}(e^{-i\lambda})},
\\  \label{eq_5_1g_fact}
\left(
{\me h^{0}_{\overline{\mu},g}(e^{i\lambda})}
\right)
\left(
{\me h^{0}_{\overline{\mu},g}(e^{i\lambda})}
\right)^{*}
&=&
(\Theta_{\overline{\mu}}(e^{-i\lambda}))^{\top}
\left \{ \beta_{ij}\gamma_{ij} ( \lambda ) \right \}_{i,j=1}^{T}
\overline{\Theta_{\overline{\mu}}(e^{-i\lambda})},
\end{eqnarray}
\begin{equation} \label{eq_5_1c_fact}
\frac{1}{2 \pi} \int_{- \pi}^{ \pi} \left|g^0_{ij}(\lambda)-g_{ij}^{1}( \lambda ) \right|d\lambda = \delta_{i}^{j},
\end{equation}
where $\vec{\alpha}_f$, $ \beta_{ij}$ are Lagrange multipliers,  functions $\left| \gamma_{ij} ( \lambda ) \right| \le 1$ and
\[
\gamma_{ij} ( \lambda )= \frac{g_{ij}^{0} ( \lambda )-g_{ij}^{1} (\lambda )}{ \left|g_{ij}^{0} ( \lambda )-g_{ij}^{1}(\lambda) \right|}: \; g_{ij}^{0} ( \lambda )-g_{ij}^{1} ( \lambda ) \ne 0, \; i,j= \overline{1,T}.
\]

For the second set of admissible spectral densities $\md D_{f0}^2 \times\md D_{g1\delta}^{2}$  we have equations
\begin{eqnarray} \label{eq_4_2f_fact}
\left(
{\me h^{0}_{\overline{\mu},f}(e^{i\lambda})}
\right)
\left(
{\me h^{0}_{\overline{\mu},f}(e^{i\lambda})}
\right)^{*} &=&
\alpha_f^{2} (\Theta_{\overline{\mu}}(e^{-i\lambda}))^{\top}\overline{\Theta_{\overline{\mu}}(e^{-i\lambda})},
\\ \label{eq_5_2g_fact}
\left(
{\me h^{0}_{\overline{\mu},g}(e^{i\lambda})}
\right)
\left(
{\me h^{0}_{\overline{\mu},g}(e^{i\lambda})}
\right)^{*} &=&
\beta^{2} \gamma_2( \lambda )(\Theta_{\overline{\mu}}(e^{-i\lambda}))^{\top}\overline{\Theta_{\overline{\mu}}(e^{-i\lambda})},
\end{eqnarray}
\begin{equation} \label{eq_5_2c_fact}
\frac{1}{2 \pi} \int_{-\pi}^{ \pi}
\left|{\mathrm{Tr}}\, (g_0( \lambda )-g_{1}(\lambda )) \right|d\lambda =\delta,
\end{equation}

\noindent where $\alpha _{f}^{2}$, $ \beta^{2}$ are Lagrange multipliers,   function $\left| \gamma_2( \lambda ) \right| \le 1$ and
\[\gamma_2( \lambda )={ \mathrm{sign}}\; ({\mathrm{Tr}}\, (g_{0} ( \lambda )-g_{1} ( \lambda ))): \; {\mathrm{Tr}}\, (g_{0} ( \lambda )-g_{1} ( \lambda )) \ne 0.\]

For the third set of admissible spectral densities $\md D_{f0}^3 \times\md D_{g1\delta}^{3}$  we have equations
\begin{eqnarray} \label{eq_4_3f_fact}
\left(
{\me h^{0}_{\overline{\mu},f}(e^{i\lambda})}
\right)
\left(
{\me h^{0}_{\overline{\mu},f}(e^{i\lambda})}
\right)^{*}
&=&(\Theta_{\overline{\mu}}(e^{-i\lambda}))^{\top}
\left\{\alpha _{fk}^{2} \delta _{kl} \right\}_{k,l=1}^{T}
\overline{\Theta_{\overline{\mu}}(e^{-i\lambda})},
\\   \label{eq_5_g_fact}
\left(
{\me h^{0}_{\overline{\mu},g}(e^{i\lambda})}
\right)
\left(
{\me h^{0}_{\overline{\mu},g}(e^{i\lambda})}
\right)^{*}
&=&
(\Theta_{\overline{\mu}}(e^{-i\lambda}))^{\top}
\left \{ \beta_{k}^{2} \gamma^2_{k} ( \lambda ) \delta_{kl} \right \}_{k,l=1}^{T}
\overline{\Theta_{\overline{\mu}}(e^{-i\lambda})},
\end{eqnarray}
\begin{equation} \label{eq_5_3c_fact}
\frac{1}{2 \pi} \int_{- \pi}^{ \pi}  \left|g^0_{kk} ( \lambda)-g_{kk}^{1} ( \lambda ) \right| d\lambda =\delta_{k},
\end{equation}

\noindent where $\alpha_{fk}^{2}$, $\beta_{k}^{2}$ are Lagrange multipliers, $\delta _{kl}$ are Kronecker symbols,  functions $\left| \gamma^2_{k} ( \lambda ) \right| \le 1$ and
\[\gamma_{k}^2( \lambda )={ \mathrm{sign}}\;(g_{kk}^{0}( \lambda)-g_{kk}^{1} ( \lambda )): \; g_{kk}^{0} ( \lambda )-g_{kk}^{1}(\lambda ) \ne 0, \; k= \overline{1,T}.\]

For the fourth set of admissible spectral densities $\md D_{f0}^4 \times\md D_{g1\delta}^{4}$  we have equations

\begin{eqnarray} \label{eq_4_4f_fact}
\left(
{\me h^{0}_{\overline{\mu},f}(e^{i\lambda})}
\right)
\left(
{\me h^{0}_{\overline{\mu},f}(e^{i\lambda})}
\right)^{*}
&=&
\alpha_f^{2} (\Theta_{\overline{\mu}}(e^{-i\lambda}))^{\top}
B_{1}
\overline{\Theta_{\overline{\mu}}(e^{-i\lambda})},
\\   \label{eq_5_3g_fact}
\left(
{\me h^{0}_{\overline{\mu},g}(e^{i\lambda})}
\right)
\left(
{\me h^{0}_{\overline{\mu},g}(e^{i\lambda})}
\right)^{*}
&=&
\beta^{2} \gamma_2'( \lambda )
(\Theta_{\overline{\mu}}(e^{-i\lambda}))^{\top}
B_{2}
\overline{\Theta_{\overline{\mu}}(e^{-i\lambda})},
\end{eqnarray}
\begin{equation} \label{eq_5_4c_fact}
\frac{1}{2 \pi} \int_{- \pi}^{ \pi}  \left| \left \langle B_{2}, g_0( \lambda )-g_{1} ( \lambda ) \right \rangle \right|d\lambda
= \delta,
\end{equation}
where $\alpha_f^{2}$,  $\beta^{2}$ are Lagrange multipliers,   function
$\left| \gamma_2' ( \lambda ) \right| \le 1$ and
\[\gamma_2' ( \lambda )={ \mathrm{sign}}\; \left \langle B_{2},g_{0} ( \lambda )-g_{1} ( \lambda ) \right \rangle : \; \left \langle B_{2},g_{0} ( \lambda )-g_{1} ( \lambda ) \right \rangle \ne 0.\]

The following theorem holds true.

\begin{theorem}
 The least favorable spectral densities $f^{0}(\lambda)$,  $g^{0}(\lambda)$  in the classes $\md D_{f0}^k \times\md D_{g1\delta}^{k}$, $k=1,2,3,4$ for the optimal linear filtering of the functional  $A\vec{\xi}$ from observations of the sequence $\vec{\xi}(m)+ \vec{\eta}(m)$ at points  $m=0,-1,-2,\ldots$  are determined
by  canonical factorizations (\ref{dd}), (\ref{fakt1}) and (\ref{fakt3}),
 equations
\eqref{eq_4_1f_fact}--\eqref{eq_5_1c_fact},  \eqref{eq_4_2f_fact}--\eqref{eq_5_2c_fact}, \eqref{eq_4_3f_fact}--\eqref{eq_5_3c_fact}, \eqref{eq_4_4f_fact}--\eqref{eq_5_4c_fact},
respectively,
 constrained optimization problem (\ref{minimax1_f_st.n_d_fact}) and restrictions  on densities from the corresponding classes $\md D_{f0}^k, \md D_{g1\delta}^{k},k=1,2,3,4$.  The minimax-robust spectral characteristic of the optimal estimate of the functional $A\vec{\xi}$ is determined by the formula (\ref{spectr A_f_st.n_d_fact}).
\end{theorem}

\subsection{Semi-uncertain filtering problem in classes $\md D_{\eps}$ of least favorable noise spectral density}

Consider a semi-uncertain filtering problem for the functional $A\vec{\xi}$, where the spectral density  $g(\lambda)$ of the stationary noise sequence  $\vec \eta(m)$ is known and $f(\lambda)$ of the  sequence with GM increments $\vec \xi(m)$  the spectral density  belongs to the
   sets of admissible spectral densities $\md D_{\eps}^k,k=1,2,3,4$
   \begin{eqnarray*}
\md D_{\varepsilon }^{1}&=&\bigg\{f(\lambda )\bigg|f(\lambda)=(1-\varepsilon )f_{1} (\lambda )+\varepsilon W(\lambda ),
\frac{1}{2\pi } \int _{-\pi}^{\pi}\frac{|\chi_{\overline{\mu}}^{(d)}(e^{-i\lambda})|^2}{|\beta^{(d)}(i\lambda)|^2}
f(\lambda )d\lambda=P\bigg\},
\\
\md D_{\varepsilon }^{2}  &=&\bigg\{f(\lambda )\bigg|{\mathrm{Tr}}\,
[f(\lambda )]=(1-\varepsilon ) {\mathrm{Tr}}\,  [f_{1} (\lambda
)]+\varepsilon {\mathrm{Tr}}\,  [W(\lambda )],
\frac{1}{2\pi} \int _{-\pi}^{\pi}\frac{|\chi_{\overline{\mu}}^{(d)}(e^{-i\lambda})|^2}{|\beta^{(d)}(i\lambda)|^2}
{\mathrm{Tr}}\,
[f(\lambda )]d\lambda =p \bigg\};
\\
\md D_{\varepsilon }^{3}  &=&\bigg\{f(\lambda )\bigg|f_{kk} (\lambda)
=(1-\varepsilon )f_{kk}^{1} (\lambda )+\varepsilon w_{kk}(\lambda),
\frac{1}{2\pi} \int _{-\pi}^{\pi}\frac{|\chi_{\overline{\mu}}^{(d)}(e^{-i\lambda})|^2}{|\beta^{(d)}(i\lambda)|^2}
f_{kk} (\lambda)d\lambda  =p_{k} , k=\overline{1,T}\bigg\};
\\
\md D_{\varepsilon }^{4} &=&\bigg\{f(\lambda )\bigg|\left\langle B_{1},f(\lambda )\right\rangle =(1-\varepsilon )\left\langle B_{2},f_{1} (\lambda )\right\rangle+\varepsilon \left\langle B_{2},W(\lambda )\right\rangle,
\frac{1}{2\pi}\int _{-\pi}^{\pi}\frac{|\chi_{\overline{\mu}}^{(d)}(e^{-i\lambda})|^2}{|\beta^{(d)}(i\lambda)|^2}
\left\langle B_{2} ,f(\lambda )\right\rangle d\lambda =p\bigg\};
\end{eqnarray*}
where  $f_{1} ( \lambda )$ is a fixed spectral density, $W(\lambda)$ is an unknown spectral density, $p,  p_k, k=\overline{1,T}$, are given numbers, $P, B_2$ are  given positive-definite Hermitian matrices.

The condition $0\in\partial\Delta_{\mathcal{D}}(f^0,g)$
implies the equations which determine the least favourable spectral densities of the noise sequence $\vec \xi(m)$. Note that the elements $\me C^{-}_{\overline{\mu},g}$ and $\me C^{+}_{\overline{\mu},g}$ are known and determined by the coefficients $\{\phi(k),k\geq0\}$ of  the canonical factorization of the spectral density matrix $g(\lambda)$.

For the first set of admissible spectral density $\md D_{\varepsilon}^{1}$ we have an equation
\begin{multline}
\left(
{\sum_{m=0}^{\infty}\ld(\overline{\psi}_{\overline{\mu}} (\me C^{-}_{\overline{\mu},g} + \me C^{+}_{\overline{\mu},g})\rd)_m e^{-i\lambda m}}
\right)
\left(
{\sum_{m=0}^{\infty}\ld(\overline{\psi}_{\overline{\mu}} (\me C^{-}_{\overline{\mu},g} + \me C^{+}_{\overline{\mu},g})\rd)_m e^{-i\lambda m}}
\right)^{*}
\\
=
\left(\sum_{k=0}^{\infty}\theta_{\overline{\mu}}(k)e^{-i\lambda k}\right)^{\top}
(\vec{\alpha}_f\cdot \vec{\alpha}_f^{*}+\Gamma(\lambda))
\left(\overline{\sum_{k=0}^{\infty}\theta_{\overline{\mu}}(k)e^{-i\lambda k}}\right), \label{eq_5_1f_fact}
\end{multline}

\noindent where $\vec{\alpha}_f$  ia a vector of Lagrange multipliers, matrix $\Gamma(\lambda )\le 0$ and $\Gamma(\lambda )=0$ if $f_{0}(\lambda )>(1-\varepsilon )f_{1} (\lambda )$.

For the second set of admissible spectral densities  $\md D_{\varepsilon}^{2}$ we have an equation
\begin{multline}
\left(
{\sum_{m=0}^{\infty}\ld(\overline{\psi}_{\overline{\mu}} (\me C^{-}_{\overline{\mu},g} + \me C^{+}_{\overline{\mu},g})\rd)_m e^{-i\lambda m}}
\right)
\left(
{\sum_{m=0}^{\infty}\ld(\overline{\psi}_{\overline{\mu}} (\me C^{-}_{\overline{\mu},g} + \me C^{+}_{\overline{\mu},g})\rd)_m e^{-i\lambda m}}
\right)^{*}
\\
=
(\alpha_f^{2} +\gamma(\lambda ))\left(\sum_{k=0}^{\infty}\theta_{\overline{\mu}}(k)e^{-i\lambda k}\right)^{\top}\left(\overline{\sum_{k=0}^{\infty}\theta_{\overline{\mu}}(k)e^{-i\lambda k}}\right),\label{eq_5_2f_fact}
\end{multline}

\noindent where  $\alpha _{f}^{2}$ is a Lagrange multiplier, function $\gamma(\lambda )\le 0$ and $\gamma(\lambda )=0$ if ${\mathrm{Tr}}\,[f_{0} (\lambda )]>(1-\varepsilon ) {\mathrm{Tr}}\, [f_{1} (\lambda )]$.

For the third set of admissible spectral densities $\md D_{\varepsilon}^{3}$, we have an equation
\begin{multline}
\left(
{\sum_{m=0}^{\infty}\ld(\overline{\psi}_{\overline{\mu}} (\me C^{-}_{\overline{\mu},g} + \me C^{+}_{\overline{\mu},g})\rd)_m e^{-i\lambda m}}
\right)
\left(
{\sum_{m=0}^{\infty}\ld(\overline{\psi}_{\overline{\mu}} (\me C^{-}_{\overline{\mu},g} + \me C^{+}_{\overline{\mu},g})\rd)_m e^{-i\lambda m}}
\right)^{*}
\\
=
\left(\sum_{k=0}^{\infty}\theta_{\overline{\mu}}(k)e^{-i\lambda k}\right)^{\top}
\left\{(\alpha_{fk}^{2} +\gamma_{k} (\lambda ))\delta _{kl} \right\}_{k,l=1}^{T}
\left(\overline{\sum_{k=0}^{\infty}\theta_{\overline{\mu}}(k)e^{-i\lambda k}}\right), \label{eq_5_3f_fact}
\end{multline}

\noindent where  $\alpha _{fk}^{2}$  are Lagrange multipliers,
 $\delta _{kl}$ are Kronecker symbols, functions $\gamma_{k}(\lambda )\le 0$ and $\gamma_{k}(\lambda )=0$ if $f_{kk}^{0}(\lambda )>(1-\varepsilon )f_{kk}^{1} (\lambda )$.

For the fourth set of admissible spectral densities $\md D_{\varepsilon}^{4}$, we have AN equation
\begin{multline}
\left(
{\sum_{m=0}^{\infty}\ld(\overline{\psi}_{\overline{\mu}} (\me C^{-}_{\overline{\mu},g} + \me C^{+}_{\overline{\mu},g})\rd)_m e^{-i\lambda m}}
\right)
\left(
{\sum_{m=0}^{\infty}\ld(\overline{\psi}_{\overline{\mu}} (\me C^{-}_{\overline{\mu},g} + \me C^{+}_{\overline{\mu},g})\rd)_m e^{-i\lambda m}}
\right)^{*}
\\
=
(\alpha_f^{2} +\gamma'(\lambda ))\left(\sum_{k=0}^{\infty}\theta_{\overline{\mu}}(k)e^{-i\lambda k}\right)^{\top}
 B_{1}
\left( \overline{\sum_{k=0}^{\infty}\theta_{\overline{\mu}}(k)e^{-i\lambda k}}\right), \label{eq_5_4f_fact}
\end{multline}

\noindent where $\alpha _{f}^{2}$ is   a Lagrange multiplier, function $\gamma' ( \lambda )\le 0$ and $\gamma' ( \lambda )=0$ if $\langle B_{1} ,f_{0} ( \lambda ) \rangle>(1- \varepsilon ) \langle B_{1} ,f_{1} ( \lambda ) \rangle$.

The following theorems  hold   true.

\begin{theorem}
 If the spectral density $g(\lambda)$ is known, the least favorable spectral density $f^{0}(\lambda)$ in the classes $\md D_{\varepsilon}^{k}$, $k=1,2,3,4$    for the optimal linear foltering of the functional  $A\vec{\xi}$ from observations of the sequence $\vec{\xi}(m)+ \vec{\eta}(m)$ at points  $m=0,-1,-2,\ldots$  is determined
by  canonical factorizations (\ref{dd}) and (\ref{fakt1}),
 equations
\eqref{eq_5_1f_fact},  \eqref{eq_5_2f_fact}, \eqref{eq_5_3f_fact}, \eqref{eq_5_4f_fact},
respectively,
 constrained optimization problem (\ref{minimax3_f_st.n_d_fact}) and restrictions  on density from the corresponding classes $\md D_{\varepsilon}^{k}$, $k=1,2,3,4$.  The minimax-robust spectral characteristic of the optimal estimate of the functional $A\vec{\xi}$ is determined by the formula (\ref{spectr A_f_st.n_d_fact}).
\end{theorem}

\section{Conclusions}

In this article, we present a solution of the filtering problem for stochastic sequences with periodically stationary   multiple seasonal increments, or sequences with periodically stationary general multiplicative (GM) increments, introduced in the article by Luz and Moklyachuk \cite{Luz_Mokl_extra_GMI}.
We propose a  solution of the filtering  problem in the case where the spectral densities of the sequence $\xi(m)$ and a noise  sequence $\eta(m)$ are exactly known.  The estimates are derived in terms of coefficients of canonical factorizations of the spectral densities, making use of results obtained in \cite{Luz_Mokl_filt_GMI} by using the Fourier transformations  of the spectral densities.
The minimax-robust approach to filtering problem is applied in the case of spectral uncertainty where densities of sequences are not exactly known while, instead, sets of admissible spectral densities allowing canonical factorizations   are specified.
We propose a representation of the mean square error in the form of a linear functional in $L_1$ with respect to spectral densities, which allows
us to solve the corresponding constrained optimization problem and describe the minimax (robust) estimates of the functionals.
Described relations   determine the least favourable spectral densities and the minimax spectral characteristics of the optimal estimates of linear functionals
for a list of specific classes of admissible spectral densities.

\section*{Appendix}

\emph{Proof of Lemma \ref{lema_fact_3}}

 Factorizations (\ref{fakt3}), (\ref{fakt2}) and Remark \ref{remark_density_adjoint} imply
\begin{eqnarray*}
\sum_{k\in\mr Z} s_{\overline{\mu}}(k)e^{i\lambda k}&=&
\frac{|\beta^{(d)}(i\lambda)|^2}{|\chi_{\overline{\mu}}^{(d)}(e^{-i\lambda})|^2}
\ld[g(\lambda)(f(\lambda)+|\beta^{(d)}(i\lambda)|^2g(\lambda))^{-1}\rd]^{\top}
=(\Psi_{\overline{\mu}}(e^{-i\lambda}))^{\top}\overline{\Psi}_{\overline{\mu}}(e^{-i\lambda})\overline{g}(\lambda)
\\
&=&\sum_{l=0}^{\infty}\psi_{\overline{\mu}}^{\top}(l)e^{-i\lambda l} \sum_{j\in\mr Z}Z_{\overline{\mu}}(j)e^{i\lambda j}
=\sum_{k\in\mr Z} \sum_{l=0}^{\infty}\psi_{\overline{\mu}}^{\top}(l)Z_{\overline{\mu}}(l+k)e^{i\lambda k}.
\end{eqnarray*}
Then
\begin{eqnarray*}
 (\Theta^{\top}_{\overline{\mu}}\me S_{\overline{\mu}}\widetilde{\me a}_{\overline{\mu}})_m&=&\sum_{j=-n(\gamma)}^{\infty}\sum_{p=m}^{\infty}\theta^{\top}_{\overline{\mu}}(p-m)S_{\overline{\mu}}(p+j+1)a_{-\mu}(j)
 \\&=&
 \sum_{j=-n(\gamma)}^{\infty}\sum_{p=m}^{\infty}\sum_{l=0}^{\infty}
\theta^{\top}_{\overline{\mu}}(p-m)\psi_{\overline{\mu}}^{\top}(l)Z_{\overline{\mu}}(l+p+j+1)a_{-\mu}(j)
\\&=&
 \sum_{j=-n(\gamma)}^{\infty}\sum_{p=m}^{\infty}\sum_{k=p}^{\infty}
(\psi_{\overline{\mu}}(k-p)\theta_{\overline{\mu}}(p-m))^{\top}Z_{\overline{\mu}}(k+j+1)a_{-\mu}(j)
\\&=&
 \sum_{j=-n(\gamma)}^{\infty}\sum_{k=m}^{\infty}\mt{diag}_q(\delta_{k,m})Z_{\overline{\mu}}(k+j+1)a_{-\mu}(j)
\\&=&
 \sum_{j=-n(\gamma)}^{\infty}Z_{\overline{\mu}}(m+j+1)a_{-\mu}(j).
 \end{eqnarray*}
The representation for  $Z_{\overline{\mu}}(j)$ follows from
\[
\sum_{j\in\mr Z}Z_{\overline{\mu}}(j)e^{i\lambda j}
=\overline{\Psi}_{\overline{\mu}}(e^{-i\lambda})\overline{g}(\lambda)
=\sum_{k\in\mr Z} \sum_{l=0}^{\infty}\overline{\Psi}_{\overline{\mu}}(l)\overline{g}_{\overline{\mu}}(l-k)e^{i\lambda k}.
\quad \square\]

\

\emph{Proof of Theorem \ref{thm3_f_st.n_d_fact}}

Under the conditions of Lemmas $\ref{lema_fact_2}$ and $\ref{lema_fact_3}$ on the spectral densities $f(\lambda)$ and $g(\lambda)$, formulas (\ref{spectr A_f_st.n_d}) and (\ref{poh A_f_st.n_d}) can be rewritten as follows. Make the following transformations:
\begin{eqnarray*}
&&\notag\frac{|\beta^{(d)}(i\lambda)|^2}
{|\chi_{\overline{\mu}}^{(d)}(e^{-i\lambda})|^2}
\ld[(f(\lambda)+|\beta^{(d)}(i\lambda)|^2g(\lambda))^{-1}\rd]^{\top}\ld(\sum_{k=0}^{\infty} \ld(\me
 P_{\overline{\mu}}^{-1}\me S_{\overline{\mu}}\me a_{\overline{\mu}}\rd)_ke^{i\lambda
 (k+1)}\rd)
 \\&=& \ld(\sum_{k=0}^{\infty}\psi_{\overline{\mu}}^{\top}(k)e^{-i\lambda k}\rd)\sum_{j=0}^{\infty}\sum_{k=0}^{\infty}
 \overline{\psi}_{\overline{\mu}}(j)(\overline{\Theta}_{\overline{\mu}}\widetilde{\me e}_{\overline{\mu}})_ke^{i\lambda(k+j+1)}
 \\&=& \ld(\sum_{k=0}^{\infty}\psi_{\overline{\mu}}^{\top}(k)e^{-i\lambda k}\rd)\sum_{m=0}^{\infty}\sum_{p=0}^{m} \sum_{k=p}^m\overline{\psi}_{\overline{\mu}}(m-k)\overline{\theta}_{\overline{\mu}}(k-p)\widetilde{e}_{\overline{\mu}}(p)e^{i\lambda (m+1)}
\\&=& \ld(\sum_{k=0}^{\infty}\psi_{\overline{\overline{\mu}}}^{\top}(k)e^{-i\lambda k}\rd)\sum_{m=0}^{\infty}\sum_{p=0}^{m}\mt{diag}_q(\delta_{m,p})\widetilde{e}_{\overline{\mu}}(m)e^{i\lambda(m+1)}
 \\&=& \ld(\sum_{k=0}^{\infty}\psi_{\overline{\overline{\mu}}}^{\top}(k)e^{-i\lambda k}\rd)\sum_{m=0}^{\infty}\widetilde{e}_{\overline{\mu}}(m)e^{i\lambda(m+1)},\end{eqnarray*}
and
\begin{eqnarray}
  \notag && \frac{|\beta^{(d)}(i\lambda)|^2}
{\chi_{\overline{\mu}}^{(d)}(e^{-i\lambda})|^2}
\ld[(f(\lambda)+|\beta^{(d)}(i\lambda)|^2g(\lambda))^{-1}\rd]^{\top}(g(\lambda))^{\top}
A(e^{-i\lambda }) (1-e^{i\lambda \mu})^n
 \\
\notag &=& \Psi_{\overline{\mu}}^{\top}(e^{-i\lambda})\overline{\Psi_{\overline{\mu}}(e^{-i\lambda})}\overline{g(\lambda)}
A(e^{-i\lambda }) (1-e^{i\lambda \mu})^n
 \\
 \notag &=&\ld(\sum_{k=0}^{\infty}\psi^{\top}_{\overline{\mu}}(k)e^{-i\lambda k}\rd)
 \sum_{m\in \mr Z}^{\infty}
\sum_{j=-n(\gamma)}^{\infty}Z_{\overline{\mu}}(m+j)a_{-\mu}(j)e^{i\lambda m}.\label{simple_sp_char_part2_f_st.n_d}
\end{eqnarray}
 Then   obtain:
\begin{eqnarray*}
\notag \vec{h}_{\overline{\mu}}(\lambda)
&=&\frac{\chi_{\overline{\mu}}^{(d)}(e^{-i\lambda})}{\beta^{(d)}(i\lambda)}
\ld(\sum_{k=0}^{\infty}\psi^{\top}_{\overline{\mu}}(k)e^{-i\lambda k}\rd)
\sum_{m=0}^{\infty}\sum_{j=-\mu n}^{\infty}\theta^{\top}_{\overline{\mu}}(l)Z_{\overline{\mu}}(j-m)a_{-\mu}(j)e^{-i\lambda m}
\\
\notag &=&\frac{\chi_{\overline{\mu}}^{(d)}(e^{-i\lambda})}{\beta^{(d)}(i\lambda)}
\ld(\sum_{k=0}^{\infty}\psi^{\top}_{\overline{\mu}}(k)e^{-i\lambda k}\rd)
\\
\notag &&\times \ld(\sum_{m=0}^{\infty}\sum_{j=0}^{\infty}Z_{\overline{\mu}}(j-m)a_{-\mu}(j)e^{-i\lambda m}+\sum_{m=0}^{\infty}\sum_{j=1}^{n(\gamma)}Z_{\overline{\mu}}(-j-m)b_{-\mu}(j)e^{-i\lambda m}\rd)
\\
\notag &=&\frac{\chi_{\overline{\mu}}^{(d)}(e^{-i\lambda})}{\beta^{(d)}(i\lambda)}
\ld(\sum_{k=0}^{\infty}\psi^{\top}_{\overline{\mu}}(k)e^{-i\lambda k}\rd)
\\
\notag && \times\ld(\sum_{m=0}^{\infty}\sum_{j=0}^{\infty}\sum_{p=m}^{\infty}
\overline{\psi}_{\overline{\mu}}(p-m)\overline{g}(p-j)a_{-\mu}(j)e^{-i\lambda m}+\sum_{m=0}^{\infty}\sum_{j=1}^{n(\gamma)}\sum_{p=m}^{\infty}
\overline{\psi}_{\overline{\mu}}(p-m)\overline{g}(p+j)b_{-\mu}(j)e^{-i\lambda m}\rd)
\\
\notag &=&\frac{\chi_{\overline{\mu}}^{(d)}(e^{-i\lambda})}{\beta^{(d)}(i\lambda)}
\ld(\sum_{k=0}^{\infty}\psi^{\top}_{\overline{\mu}}(k)e^{-i\lambda k}\rd)
\sum_{m=0}^{\infty}( (\widetilde{\Psi}_{\overline{\mu}})^*\me G^{-}\me a_{-\mu} +(\widetilde{\Psi}_{\overline{\mu}})^* \me G^{+}\me b_{-\mu})_m e^{-i\lambda m}
\\
 &=&\frac{\chi_{\overline{\mu}}^{(d)}(e^{-i\lambda})}{\beta^{(d)}(i\lambda)}
\ld(\sum_{k=0}^{\infty}\psi^{\top}_{\overline{\mu}}(k)e^{-i\lambda k}\rd)
\sum_{m=0}^{\infty}(\overline{\psi}_{\overline{\mu}} \me C^{-}_{\overline{\mu},g} +\overline{\psi}_{\overline{\mu}} \me C^{+}_{\overline{\mu},g})_m e^{-i\lambda m}.
\end{eqnarray*}

The value of the mean square error $\Delta(f,g;\widehat{A}\xi)$ is calculated by the formula
\begin{eqnarray*}
  \notag \Delta\ld(f,g;\widehat{A}\xi\rd)&=&\Delta\ld(f,g;\widehat{A}\eta\rd)= \mt E\ld|A\eta-\widehat{A}\eta\rd|^2
 \\
 \notag &=& \frac{1}{2\pi}\int_{-\pi}^{\pi}(\vec A(e^{i\lambda}))^{\top}g(\lambda)\overline{\vec A(e^{i\lambda})}d\lambda
 +
 \frac{1}{2\pi}\int_{-\pi}^{\pi}(\vec h_{\overline{\mu}}(\lambda))^{\top}(f(\lambda)+|\beta^{(d)}(i\lambda)|^2g(\lambda))\overline{\vec h_{\overline{\mu}}(\lambda)}d\lambda
 \\
 \notag && -\frac{1}{2\pi}\int_{-\pi}^{\pi}(\vec h_{\overline{\mu}}(\lambda))^{\top}
 \beta^{(d)}(i\lambda)g(\lambda)\overline{A(e^{-i\lambda})}d\lambda
 -
 \frac{1}{2\pi}\int_{-\pi}^{\pi}(A(e^{-i\lambda}))^{\top}\overline{\beta^{(d)}(i\lambda)}
 g(\lambda)\overline{\vec h_{\overline{\mu}}(\lambda)}d\lambda
 \\
 &=&\|\widetilde{\Phi}\me a\|^2-\|\overline{\psi}_{\overline{\mu}}( \me C^{-}_{\overline{\mu},g} + \me C^{+}_{\overline{\mu},g})\|^2.
\quad \square \end{eqnarray*}

\end{document}